\documentclass[11pt]{amsart}
\usepackage{amsmath,amssymb,amsfonts}
\usepackage{graphicx,xcolor}
\usepackage{epstopdf}
\usepackage[colorlinks=true]{hyperref}
\hypersetup{urlcolor=blue, citecolor=red}

\usepackage[all,cmtip]{xy}

\renewcommand{\subjclassname}{\textup{2010} Mathematics Subject Classification}
\newtheorem{theorem}{Theorem}[section]

\newtheorem{lemma}[theorem]{Lemma}
\newtheorem{proposition}{Proposition}

\theoremstyle{definition}
\newtheorem{definition}[theorem]{Definition}
\newtheorem{remark}{Remark}

\title[Filled Julia set of some class of H\'enon-like maps] 
      {Filled Julia set of some \\ class of H\'enon-like maps}

\author[Danilo Antonio Caprio]{}


\email{danilo.caprio@unesp.br}

\thanks{Supported by \textit{FAPESP} grants $2015/26161-6$ and $2018/13720-5$}

\begin{document}
\maketitle

\centerline{\scshape Danilo Antonio Caprio}
\medskip
{\footnotesize
\centerline{UNESP - Departamento de Matem\'atica do}
\centerline{Instituto de Bioci\^encias, Letras e Ci\^encias Exatas.}
\centerline{Rua Crist\'ov\~ao Colombo, 2265, Jardim Nazareth, 15054-000}
\centerline{S\~ao Jos\'e do Rio Preto, SP, Brasil.}
}

\bigskip

\begin{abstract}
In this work we consider a class of endomorphisms of $\mathbb{R}^2$ defined by $f(x,y)=(xy+c,x)$, where $c\in\mathbb{R}$ is a real number and we prove that when $-1<c<0$, the forward filled Julia set of $f$ is the union of stable manifolds of fixed and $3-$periodic points of $f$. We also prove that the backward filled Julia set of $f$ is the union of unstable manifolds of the saddle fixed and $3-$periodic points of $f$.
\end{abstract}

\section{Introduction}

Let $f:\mathbb{C}\longrightarrow \mathbb{C}$ be a holomorphic map.
The forward filled Julia set associated to $f$ (or simply filled Julia set) is by definition the set
$\mathcal{K}^+(f)=\{z\in\mathbb{C}: (f^n(z))_{n\geq 0} \textrm{ is bounded}\}$, where $f^n$
is the $n-$th iterate of $f$. This definition can be to extended to polynomial maps defined in $\mathbb{C}^d$, for $d\geq 2$.

Filled Julia sets and its boundary (called Julia set) have many topological and dynamical properties.
These sets were defined independently by Fatou and Julia (see \cite{fatou19} and \cite{fatou21}, \cite{julia18} and \cite{julia22}) and they are associated with many areas as dynamical systems, number theory, topology and functional analysis (see \cite{dv}).

Higher dimensional Julia sets were studied by many mathematicians. In particular, when $f$ is a H\'enon map defined by
$f(x,y)=(y,P(y)+c-ax)$, for all $(x,y)\in \mathbb{C}^2$, where $P$ is a complex polynomial function and $a,c$ are fixed complex numbers, many topological and dynamical properties of the forward and backward Julia set associated to $f$ have been proved (\cite{bs}, \cite{fs}). For example, in \cite{hw} the authors considered the map $H_a:\mathbb{R}^2\longrightarrow \mathbb{R}^2$ defined by $H_a(x,y)=(y,y^2+ax)$ where $0<a<1$ is given and they proved that $\mathcal{K}(H_a)=\{\alpha,p\}\cup[W^s(\alpha)\cap W^u(p)]$, where $\alpha=(0,0)$ is the attracting fixed point of $H_a$, $p=(1-a,1-a)$ is the repelling fixed point of $H_a$ and $\mathcal{K}(H_a):=\mathcal{K}^+(H_a)\cap \mathcal{K}^-(H_a)$ is the Julia set associated $H_a$.

Another interesting class of Julia sets was considered in \cite{g}. Precisely, the author considered a family of maps $f_{\alpha,\beta}:\mathbb{C}^2\longrightarrow\mathbb{C}^2$ defined by $(x,y)\mapsto(xy+\alpha,x+\beta)$ with $\alpha,\beta\in\mathbb{C}$ and he proved, among other things, that the family $f_{\alpha,\beta}$ have a measure of the maximal entropy $\frac{1+\sqrt{5}}{2}$.  In \cite{ms}, the authors considered $f_\alpha=f_{\alpha,0}$: they  extended the Killeen and Taylor stochastic adding machine in base $2$ (see \cite{killentaylor}) to the Fibonacci base and they proved that the spectrum of the transition operator associated with this stochastic adding machine is related to the set $\mathcal{K}^+(f_\alpha)$, where $\alpha\in\mathbb{R}$ is a real value. In \cite{am}, the authors proved that the spectrum of theses operators in other Banach spaces contains the set $\{z\in\mathbb{C}^2:(z,z)\in\mathcal{K}^+(f_\alpha)\}$. In \cite{abms}, the authors studied many topological properties of the set $\mathcal{K}^+(f_\alpha)\cap G(h)$ where $h$ is a non null polynomial function and $G(h)$ is the graph of $h$ in $\mathbb{C}^2$. In particular, they proved that $\mathcal{K}^+(f_\alpha)\cap G(h)$ is a quasi-disk, if $|\alpha|$ is small and $h(z)=z$. A more general class is given in \cite{cmv}, where the authors defined  the stochastic Vershik map related to a stationary ordered Bratteli diagram with incidence matrix $\left(\begin{array}{cc} a & b \\ c & d\end{array} \right)$, $a,b,c,d\in\mathbb{N}$, and they proved that the spectrum of the transition operator associated with this is connected to the set $\mathcal{K}^+(f_{\alpha,a,b,c,d})$, where $f_{\alpha,a,b,c,d}:\mathbb{C}^2\longrightarrow\mathbb{C}^2$ is a family of maps defined by $(x,y)\mapsto(x^ay^b+\alpha,x^cy^d+\alpha)$ with $\alpha\in\mathbb{R}$.

In this paper, we will study the set $\mathcal{K}^+=\mathcal{K}^+(f)$, where
$f=f_c:\mathbb{R}^2\longrightarrow\mathbb{R}^2$ is defined by
$f(x,y)=(xy+c,x)$ and $c\in\mathbb{R}$, is a real number with $-1<c<0$.
This set was studied by S. Bonnot, A. de Carvalho and A.
Messaoudi in \cite{bcm}, in the case where $0\leq c<\frac{1}{4}$. They proved that
$\mathcal{K}^+$ is the union of stable manifolds of the fixed and $3$-periodic points of $f$. Furthermore,
they proved that the backward filled Julia set defined by
$\mathcal{K}^-:=\{(x,y)\in\mathbb{R}^2:f^{-n}(x,y) \textrm{ is defined for all } n\in\mathbb{N} \textrm{ and } (f^{-n}(x,y))_{n\geq 0} \textrm{
	is bounded}\}$ is the union of unstable manifolds of the saddle fixed and $3$-periodic points of $f$.

Here, we will extend the results of \cite{bcm} to  the case $-1< c <0$. The extension can not be deduced directly from the study of the case $c\in \left[0,\frac{1}{4}\right]$ since the dynamic are different. We will discuss these differences later in the paper.

The paper is organized as follows. In Section \ref{secaokmais} we prove that for $-1<c<0$, $\mathcal{K}^+$ is the union of stable manifolds of the fixed and $3$-periodic points of $f$. For this, we need to give a filtration of $\mathbb{R}^2$ and from Proposition \ref{propinf} we obtain the subsets of this filtration that are not contained in $\mathcal{K}^+$. In Subsection \ref{subsecr0r1r2r3} we give the subsets of the filtration that are contained in the stable manifold of the attracting fixed point of $f$. In Subsection \ref{subsecwstheta} (and \ref{subsecwsp}) we give the description of stable manifold of the saddle fixed point (and $3$-periodic points) of $f$. Section \ref{secaokmenos} is devoted to prove that when $-1<c<0$, the backward filled Julia set $\mathcal{K}^-$ is the union of unstable manifolds of the saddle fixed and $3$-periodic points of $f$.

\section{Description of $\mathcal{K}^+$}
\label{secaokmais}

Let $c\in \mathbb{R}$ be a real number and consider $f:\mathbb{R}^2\longrightarrow\mathbb{R}^2$ the map defined by $f(x,y)=(xy+c,x)$.

The goal of this section is to study the filled Julia set
\begin{center} 
	$\mathcal{K}^+=\{(x,y)\in\mathbb{R}^2:(f^n(x,y))_{n\geq 0} \textrm{ is bounded}\}$.
\end{center}
\begin{figure}[!h]
	\centering
	\includegraphics[scale=0.29]{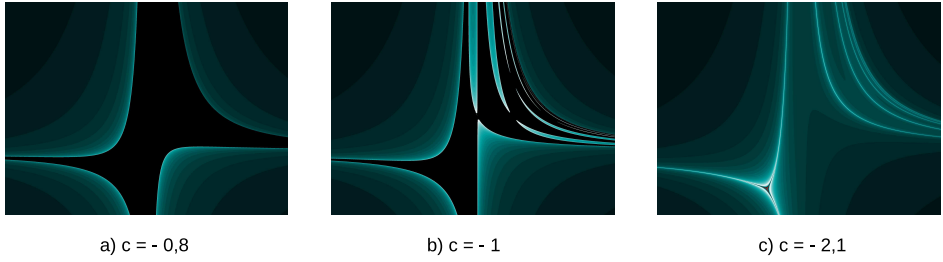}
	\caption{The set $\mathcal{K}^+$ in the cases: $(a)$ $c=-0.8$, $(b)$ $c=-1$, $(c)$ $c=-2.1$.} \label{kmais}
\end{figure}

It is easy to check that $\alpha=(a_1,a_1)$ and $\theta=(a_2,a_2)$ are the two fixed points of $f$, where $a_1=\frac{1-\sqrt{1-4c}}{2}$ and $a_2=\frac{1+\sqrt{1-4c}}{2}$. Obviously, $\alpha,\theta\in \mathbb{R}^2$ if and only if $c\leq \frac{1}{4}$. Also, $p=(-1,-1)$, $f(p)=(1+c,-1)$ and $f^2(p)=(-1,1+c)$ is the only $3-$cycle of $f$. Furthermore, the following properties are valid:
	
	\noindent $(a)$ if $c<\frac{1}{4}$, then $\theta$ is a saddle fixed point of $f$;
	
	\noindent $(b)$ if $-2<c<\frac{1}{4}$, then $\alpha$ is an attracting fixed point  of $f$;
	
	\noindent $(c)$ if $c=-2$, then $\alpha$ is an indifferent fixed point of $f$ with eigenvalues $e^{\pm \frac{2\pi i}{3}}$;
	
	\noindent $(d)$ if $c<-2$, then $\alpha$ is a repelling fixed point of $f$;
	
	\noindent $(e)$ if $c=\frac{1}{4}$, then $\alpha=\theta=\left(\frac{1}{2},\frac{1}{2}\right)$ and the corresponding eigenvalues are $1$ and $-\frac{1}{2}$.

For the proof of $(a)$, $(b)$, $(c)$, $(d)$ and $(e)$, we refer to \cite{bcm} (Proposition 3.1).

\begin{definition}
For each $\tau\in\{\alpha,\theta,p,f(p),f^2(p)\}$, let $W^s(\tau)$ and $W^u(\tau)$ the stable and unstable manifolds (respectively) of $f$ defined by
$$
\begin{array}{lr}
W^s(\tau)=\{z\in\mathbb{R}^2:\lim_{n\to +\infty} d(f^n(z),f^n(\tau))=0\} & and  \\
W^u(\tau)= \{z\in\mathbb{R}^2:\lim_{n\to +\infty} d(f^{-n}(z),f^{-n}(\tau))=0\}. &
\end{array}
$$
\end{definition}

When $c=0$, we have that $\alpha=(0,0)$, $\theta=(1,1)$, $p=(-1,-1)$, $f(p)=(1,-1)$ and $f^2(p)=(-1,1)$. In Figure \ref{casoc0} we can see $\mathcal{K}^+$ (when $c=0$), where $W^s(\alpha)$ is the gray region, i.e. $W^s(\alpha)=int(\mathcal{K}^+)$, $W^s(\theta)$ is the curve which contains the point $\theta$ and $W^s(p)\cup W^s(f(p))\cup W^s(f^2(p))$  are the curves which contains the points $p$, $f(p)$ and $f^2(p)$.

\begin{figure}[!h]
	\centering
	\includegraphics[scale=0.5]{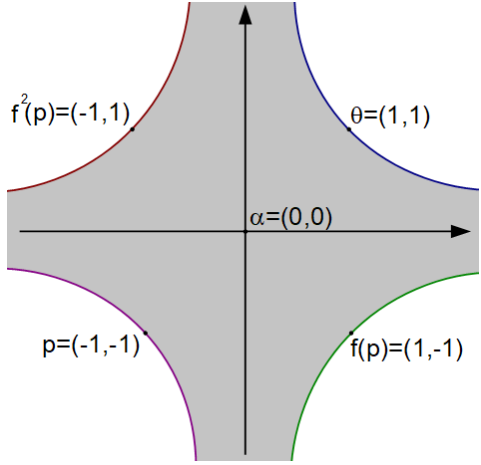}
	\caption{The set $\mathcal{K}^+$ when $c=0$.} \label{casoc0}
\end{figure}

In this section, we will concentrate on proving the following theorem.

\begin{theorem}
	If $-1< c<0$, then $\mathcal{K}^+=W^s(\alpha)\cup W^s(\theta)\cup
	W^s(p)\cup W^s(f(p))\cup W^s(f^2(p))$, where $p=(-1,-1)$,
	$\alpha=(a_1,a_1)$ and $\theta=(a_2,a_2)$. \label{teo34}
\end{theorem}

For the proof of Theorem \ref{teo34}, we use the following notations and filtration of the plane (see Figure \ref{figlmnp}).

\begin{eqnarray} \label{notacao}
\begin{array}{lcl}
L=[a_2,+\infty [ \times [a_2,+\infty [, & \hspace{.3cm} & \mathcal{R}=\mathcal{R}'\cup R_0, \\
M=]-\infty,-1 ] \times [1+c,+\infty [,  & & A=[0,a_2]\times[a_2,+\infty[, \\
N=]-\infty,-1] \times ]-\infty,-1 ],    & & B=[-1,0]\times[1+c,+\infty[, \\
P=[1+c,+\infty [ \times ]-\infty,-1],   & & C=]-\infty,-1]\times[0,1+c],  \\
S'=L\cup M\cup N\cup P,                 & & D=]-\infty,-1]\times[-1,0], \\
S=S'\setminus\{\theta,p,f(p),f^2(p)\},  & & E=[-1,0]\times]-\infty,-1], \\
R_0=[0,1+c]\times [0,a_2],              & & F=[0,1+c]\times]-\infty,-1], \\
R_1=[-1,0]\times [0,1+c],               & & G=[1+c,+\infty[\times[-1,0],  \\
R_2=[-1,0]\times [-1,0],                & & H_1=[1+c,a_2]\times[0,a_2],  \\
R_3=[0,1+c]\times [-1,0],               & & H_2=[a_2,+\infty[\times[0,a_2]. \\
\mathcal{R}' =R_1 \cup R_2\cup R_3, & & 
\end{array}
\end{eqnarray}

\begin{figure}[!h]
	\centering
	\includegraphics[scale=0.6]{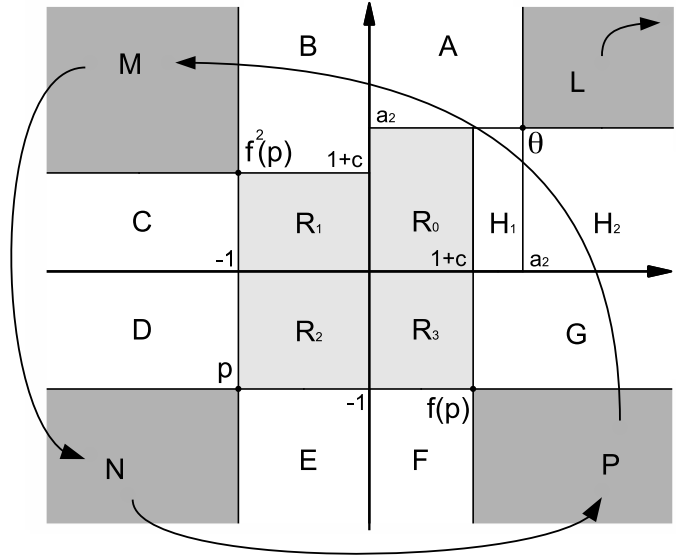}
	\caption{Filtration of $\mathbb{R}^2$ for $f$ in the case $-1<c<0$.} \label{figlmnp}
\end{figure}

In the proof of Theorem \ref{teo34}, we will start by proving that $S\subset \mathbb{R}^2\setminus \mathcal{K}^+$.

After that, we will prove that $\mathcal{R}\setminus \{p,f(p),f^2(p)\}\subset W^s(\alpha)$ and $W^s(\theta)=\bigcap_{n=0}^{+\infty}f^{-n}(A\cup H_2)$, i.e. $z$ is a point of $W^s(\theta)$ if and only if its orbit is contained in the set $A\cup H_2$.

Finally, we end up by proving that $\mathcal{K}^+\setminus W^s(\alpha)\cup W^s(\theta)=W^s(p)\cup W^s(f(p))\cup W^s(f^2(p))$.

In other words, we need the following Proposition. Its proof is similar to that given in \cite{bcm} (Proposition 3.5, pages 629, 630 and 631).

\begin{proposition}
If $-1< c<0$, then the following properties are valid:
	
	\noindent $1.$ $f(L)\subset L$, $f(M)\subset N$, $f(N)\subset P$ and $f(P)\subset M$.
	
	\noindent $2.$ For all $(x,y)\in S$, $\|f^n(x,y)\|_\infty$ diverges to $+\infty$ as $n$ goes to $+\infty$.
	\label{propinf}
\end{proposition}

\begin{proposition}
If $-1<c<0$ and $z\in \mathcal{R}\setminus \{p,f(p),f^2(p)\}$, then $\displaystyle\lim_{n\to+\infty}f^n(z)=(a_1,a_1)=\alpha$. \label{propostion2}
\end{proposition}

A difference between this work and \cite{bcm} is Proposition \ref{propostion2}. Here, we have that $\mathcal{R}\setminus\{p,f(p),f^2(p)\}\subset W^s(\alpha)\subset \mathcal{K}^+$, with $f(\mathcal{R})\subset \mathcal{R}$. In \cite{bcm}, $\mathcal{R}\cap W^s(\alpha)\neq\emptyset$, $\mathcal{R}\cap W^s(\theta)\neq\emptyset$ and $\mathcal{R}\setminus\mathcal{K}^+\neq \emptyset$. Furthermore, here we will see that the points of $W^s(\alpha)$ converges to $\alpha$ in a spiral way, while in \cite{bcm} the points of $W^s(\alpha)$ converges to $\alpha$ in a  monotonous way.

The proof of Proposition \ref{propostion2} will be displayed in Subsection \ref{subsecr0r1r2r3}.

\begin{proposition} If $-1<c<0$, then \label{proposition3} $W^s(\theta)= \bigcap_{n=0}^{+\infty}f^{-n}(A\cup H_2)$.
\end{proposition}

Proposition \ref{proposition3} means that $W^s(\theta)$ is characterized by the set of points in $A\cup H_2$ that never leave $A\cup H_2$ after each iteration of the map $f$. Its proof will be displayed in Subsection \ref{subsecwstheta}.

\begin{proposition} \label{proposition4}
If $-1<c<0$, then $\mathcal{K}^+\setminus W^s(\alpha)\cup W^s(\theta)=W^s(p)\cup W^s(f(p))\cup W^s(f^2(p))$.
\end{proposition}

In the proof of Proposition \ref{proposition4}, we will consider the points in $\mathcal{K}^+$ (therefore, the orbits never enter in $S$) whose the orbits never enter in $\mathcal{R}\setminus \{p,f(p),f^2(p)\}$ and whose the orbit is not contained in the set $A\cup H_2$. This proof will be given in Subsection \ref{subsecwsp}.

\subsection{Proof of Proposition \ref{propostion2}} \label{subsecr0r1r2r3}
Let $Y=[c,0]\times[c,0]$ and note that $f(Y)\subset [c,c^2+c]\times[c,0]\subset Y$. In order to prove Proposition \ref{propostion2}, we will prove that if $z\in\mathcal{R}\setminus\{p,f(p),f^2(p)\}$, then there exists $N\in\mathbb{N}$ such that $f^n(z)\in Y$,  for all $n\geq N$.
For that, we need the following.

\begin{lemma}
	The following properties are valid:
	
	\noindent $1.$ $f(R_0)\subset R_0\cup R_1$.
	
	\noindent $2.$ $f(R_1)\subset R_2$.
	
	\noindent $3.$  $f(R_2)\subset R_2\cup R_3$.
	
	\noindent $4.$ $f(R_3)\subset R_1$. \label{lema3}
\end{lemma}
\begin{proof} $1.$ We have that $f(R_0)\subset [c,a_2(1+c)+c]\times [0,1+c]$. On the other hand, $a_2(1+c)+c\leq 1+c$ if and only if $c\geq-2$. Since $-1<c<0$, it follows that $f(R_0)\subset [c,1+c]\times[0,1+c]\subset R_0\cup R_1$.
	
	\noindent $2.$ $f(R_1)\subset [-1,c]\times [-1,0]\subset R_2$.
	
	\noindent $3.$ $f(R_2)\subset [c,1+c]\times[-1,0]\subset R_2\cup R_3$.
	
	\noindent $4.$ $f(R_3)\subset [-1,c]\times[0,1+c]\subset R_1$.
\end{proof}

\begin{remark} \label{obsrlinha} 
From Lemma \ref{lema3}, we deduce that $f(\mathcal{R}')\subset \mathcal{R}'$ and $\mathcal{R}=\mathcal{R}'\cup R_0\subset \mathcal{K}^+$.
\end{remark}

\begin{lemma}
There exists $N\in\mathbb{N}$ such that $f^n(\mathcal{R})\subset \mathcal{R}'$, for all $n\geq N$. \label{lema33}
\end{lemma}
\begin{proof} From Lemma \ref{lema3}, we have $f^n(R_0)\setminus \mathcal{R}'=f^n(R_0)\cap R_0$, for all $n\in\mathbb{N}$. Observe that
	$$
	\begin{array}{rcl}
f(R_0)\setminus \mathcal{R}' & \subset & [0,1+c]\times[0,1+c], \\
f^2(R_0)\setminus \mathcal{R}' & \subset & [0,(1+c)^2+c]\times[0,1+c], \\
f^3(R_0)\setminus \mathcal{R}' & \subset & [0,(1+c)^3+c(1+c)+c]\times[0,(1+c)^2+c] \\
                     & \subset & [0,(1+c)^3+c]\times[0,(1+c)^2], \\
f^4(R_0)\setminus \mathcal{R}' & \subset & [0,(1+c)^5+c(1+c)^2+c]\times[0,(1+c)^3+c] \\
                     & \subset & [0,(1+c)^5+c]\times[0,(1+c)^3]. \\
	\end{array}
	$$
	
	Continuing in this way, we have that
	
	\begin{center}
		$f^n(R_0)\setminus \mathcal{R}' \subset [0,(1+c)^{F_{n-1}}+c]\times[0,(1+c)^{F_{n-2}}]$, for all $n\geq 4$,
	\end{center}
	where $(F_n)_{n\geq 0}$ is the Fibonacci sequence defined by $F_0=1$, $F_1=2$ and $F_n=F_{n-1}+F_{n-2}$, for all $n\geq 2$.

	Since $0<1+c<1$, it follows that there exists $N\in \mathbb{N}$ such that $(1+c)^{F_{N-1}}+c<0$, i.e. $f^N(R_0)\setminus \mathcal{R}'=\emptyset$, which implies that $f^N(R_0)\subset \mathcal{R}'$. Thus, $f^n(\mathcal{R})\subset\mathcal{R}'$, for all $n\geq N$, from Remark \ref{obsrlinha}.
\end{proof}

Now, we will prove that if $z\in \mathcal{R}'\setminus\{p,f(p),f^2(p)\}$, then there exists $N\in\mathbb{N}$ such that $f^n(z)\in Y$, for all $n\geq N$. For this, according to Lemmas \ref{lema3} and \ref{lema33} it suffices to consider $z\in R_2\setminus\{p\}$.

\begin{lemma}
	If $z\in R_2\setminus \{p\}$, then there exists $N\in\mathbb{N}$ such that $f^n(z)\in Y$, for all $n\geq N$. \label{prop66}
\end{lemma}
\begin{proof}
If $z\in Y$, there is nothing to prove. Let $z=(x,y)\in R_2\setminus Y\cup \{p\}$.

\smallskip

\noindent \textbf{Claim:} If $z,f(z)\in R_2$, then $f^n(z)\in Y$, for all $n\geq 2$.

In fact, let $z\in R_2$ such that $f(z)\in R_2$. Hence,
\begin{center}
	$f(z)\in f(R_2)\cap R_2\subset [c,1+c]\times[-1,0]\cap R_2\subset [c,0]\times [-1,0]$.
\end{center}

Thus, $f^2(z)\in f([c,0]\times [-1,0])\subset [c,0]\times[c,0]=Y$ and since $f(Y)\subset Y$, it follows that $f^n(z)\in Y$, for all $n\geq 2$ and the proof of the
 	Claim is complete.

\smallskip

For the proof of Lemma \ref{prop66}. we need to consider two cases:

\textbf{Case 1:} $z\in ([-1,c]\times [c,0])\cup ([c,0]\times[-1,c])$.

\noindent Since $f([-1,c]\times [c,0])\subset [c,0]\times[-1,c]\subset R_2$ and $f([c,0]\times[-1,c])\subset [c,0]\times[c,0]\subset R_2$, it follows from the Claim that
$f^n(z)\in Y$, for all $n\geq 2$.

\textbf{Case 2:} $z\in[-1,c[\times[-1,c[\setminus \{p\}$.

\noindent \textbf{Case 2.1:} $\min\{x,y\}>-1$.

Suppose for each $N\in\mathbb{N}$, $f^N(z)\notin Y$. Thus, by Claim and Lemma \ref{lema3}, we have
\begin{equation}
f^{3n}(x,y)\in R_2, \textrm{ } f^{3n+1}(x,y)\in R_3 \textrm{ and } f^{3n+2}(x,y)\in R_1, \label{1}
\end{equation}
for all integer $n\in\mathbb{N}$.

Let $(x_n,y_n)=f^n(x,y)$. Observe that if there exists $n\in\mathbb{N}$ such that $(x_{3n},y_{3n})\in [-\sqrt{-c},c]\times[-\sqrt{-c},c]$, then $(x_{3n+1},y_{3n+1})\in [c^2+c,0]\times[-\sqrt{-c},c]$ $\subset [c,0]\times[-1,c]$, and from Case 1 it follows that $(x_{3n+k},y_{3n+k})\in Y$, for all $k\geq 3$.

Therefore, we need to suppose that 
\begin{equation}
\min\{x_{3n},y_{3n}\}<-\sqrt{-c}, \textrm{ for all } n\in\mathbb{N}. \label{2}
\end{equation}
Let $z_0=\min\{x,y\}$. Thus, from relation (\ref{2}), we have $z_0<-\sqrt{-c}$, which implies that $z_0^2+c>0$. Hence,
$(x,y)\in[z_0,c]\times[z_0,c]$, $(x_1,y_1)\in[c^2+c,z_0^2+c]\times[z_0,c]$ and from relation (\ref{1}), it follows that

\begin{center}
	$(x_1,y_1)\in[0,z_0^2+c]\times[z_0,c]$, $(x_2,y_2)\in[z_0(z_0^2+c)+c,c]\times[0,z_0^2+c]$ and $(x_3,y_3)\in[(z_0(z_0^2+c)+c)(z_0^2+c)+c,c]\times[z_0(z_0^2+c)+c,c]$.
\end{center}

Since $-1<z_0<c<a_1<0$, it is not hard to prove that $z_0(z_0^2+c)+c\leq (z_0(z_0^2+c)+c)(z_0^2+c)+c$, i.e. $(x_3,y_3)\in[z_0(z_0^2+c)+c,c]\times[z_0(z_0^2+c)+c,c]$.

Let  $g:\mathbb{R}\longrightarrow\mathbb{R}$ defined by $g(x)=x(x^2+c)+c$ and define $z_{n+1}=g(z_n)$, for all $n\in\mathbb{N}$.

It is also easy to check that $g(z_0)\geq z_0$. Therefore, we have that $z_1=g(z_0)\geq z_0$ and by definition of $g$, it follows that $(x_3,y_3)\in[z_1,c]\times[z_1,c]$.

In the same way, we can prove that $(x_{3k},y_{3k})\in[z_k,c]\times[z_k,c]$, with $-1<z_0<z_1<\ldots<z_k<-\sqrt{-c}$ and $\min\{x_{3k},y_{3k}\}<-\sqrt{-c}$, for all $k\in\mathbb{N}$. Hence, $(z_n)_n$ is a bounded increasing sequence. Therefore, it admits limit $l$. We thus have $g(l)=l$ with $l\in\{-1,a_1,a_2\}$ which is impossible, since $-1<z_0<z_1<\ldots<z_k<\ldots<l\leq -\sqrt{-c}<c<a_1<a_2$. Therefore, there exists $N\in\mathbb{N}$ such that $f^N(z)\in Y$.

\noindent \textbf{Case 2.2:} $x=-1$ and $-1<y<c$.

We have $f(x,y)=(-y+c,-1)\in R_3$, $f^2(x,y)=(y,-y+c)\in R_1$ and $f^3(x,y)=(-y^2+cy+c,y)\in R_2$. Let $(x_0,y_0)=f^3(x,y)$. Thus, $(x_0,y_0)\in R_2$, with $x_0>-1$ and $y_0>-1$ and by Case $2.1$ we are done.

\noindent \textbf{Case 2.3:} $-1<x<c$ and $y=-1$.

We have $f(x,y)=(-x+c,x)\in R_3$, $f^2(x,y)=(-x^2+cx+c,-x+c)\in R_1$ and $f^3(x,y)=((-x^2+cx+c)(-x+c)+c, -x^2+cx+c)\in R_2$. Let $(x_0,y_0)=f^3(x,y)$. Thus, $(x_0,y_0)\in R_2$, with $x_0>-1$ and $y_0>-1$ and by Case $2.1$ we are done.
\end{proof}

Now, to complete the proof of Proposition \ref{propostion2}, we will prove that if $z\in Y$, then $\displaystyle\lim_{n\to\infty}f^n(z)=\alpha$. For this, we need to consider the following filtration of $Y$: let $R=[a_1,0]\times[c,a_1]$, $S=[a_1,0]\times[a_1,0]$, $T=[c,a_1]\times[a_1,0]$ and $U=[c,a_1]\times[c,a_1]$ (see Figure \ref{figrstuy}).
\begin{figure}[h!]
	\centering
	\includegraphics[scale=0.5]{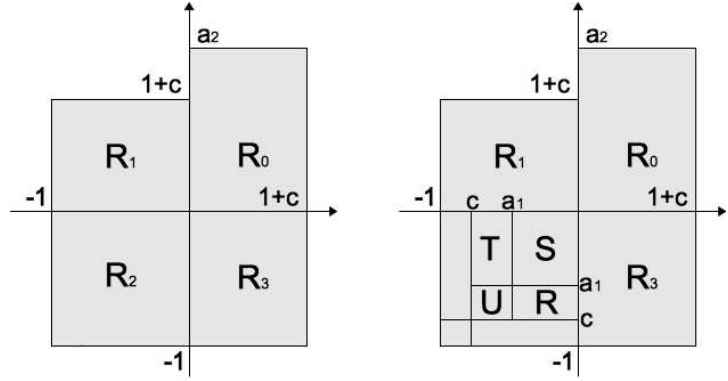}
	\caption{Rectangles $R$, $S$, $T$, $U$.} \label{figrstuy}
\end{figure}

\begin{lemma}
	The following results are valid:
	
	\noindent $1.$ $f(R)\subset S\cup T$.
	
	\noindent $2.$ $f(S)\subset T$.
	
	\noindent $3.$ $f(T)\subset U\cup R$.
	
	\noindent $4.$ $f(U)\subset R$. \label{lema4}
\end{lemma}
\begin{proof} $1.$ $f(R)\subset [c,ca_1+c]\times [a_1,0]$. Since $-1<c<a_1<0$, it follows that $ca_1+c<0$ and $ca_1+c\leq c^2+c<0$. Thus, $f(R)\subset S\cup T$.
	
	\noindent $2.$ $f(S)\subset [c,a_1]\times [a_1,0]= T$.
	
	\noindent $3.$ $f(T)\subset [c,ca_1+c]\times[c,a_1]\subset [c,0]\times[c,a_1]= U\cup R$.
	
	\noindent $4.$ $f(U)\subset [a_1,c^2+c]\times[c,a_1]\subset [a_1,0]\times[c,a_1]= R$.
\end{proof}

\begin{lemma}
	Let $-\frac{3}{4}< c<0$. If $z\in Y$, then $\displaystyle\lim_{n\to+\infty}f^n(z)=(a_1,a_1)$. \label{lemaazia}
\end{lemma}
\begin{proof}
	Since $z\in[c,0]\times[c,0]$, it follows that $f(z)\in [c,c^2+c]\times [c,0]$ and $f^2(z)\in [c,c^2+c]\times [c,c^2+c]$.
	
	Let $(c_n)_{n\geq 0}$ be the sequence defined by $c_0=c$ and $c_n=c_{n-1}^2+c$, for all $n\geq 1$. Thus, $f^2(z)\in [c_0,c_1]\times[c_0,c_1]$.
	
\noindent \textbf{Claim:} $f^{4n+2}(z)\in [c_{2n},c_{2n+1}]\times[c_{2n},c_{2n+1}]$ with $c\leq c_{2n}<c_{2n+2}<a_1<c_{2n+3}<c_{2n+1}<0$, for all $n\geq 0$.

 In fact, it is not hard to prove that $c_0<c_2<a_1<c_3<c_1<0$.
	
Let $k\geq 0$ and by induction, suppose that  $f^{4i+2}(z)\in [c_{2i},c_{2i+1}]\times[c_{2i},c_{2i+1}]$ for all $i\in\{0,\ldots, k\}$, with $c\leq c_{2i}<c_{2i+2}<a_1<c_{2i+3}<c_{2i+1}<0$, for all $i\in\{0,\ldots, k-1\}$.
	
	Hence, since $c\leq c_{2k-2}<c_{2k}<a_1<c_{2k+1}<c_{2k-1}< 0$, by definition of $(c_n)_{n\geq 0}$ it follows that
\begin{center}	
$c<c_{2k}<c_{2k+2}<a_1<c_{2k+3}<c_{2k+1}< 0$
\end{center}
 and since $f^{4k+2}(z)\in [c_{2k},c_{2k+1}]\times[c_{2k},c_{2k+1}]$, it follows that
	$$
	\begin{array}{rcl}
	f^{4k+3}(z) & \in & [c_{2k+2},c_{2k+1}]\times[c_{2k},c_{2k+1}] \textrm{, } \\
	f^{4k+4}(z) & \in & [c_{2k+2},c_{2k}c_{2k+2}+c]\times[c_{2k+2},c_{2k+1}]\subset \\
	            &     & [c_{2k+2},c_{2k+1}]\times[c_{2k+2},c_{2k+1}] \textrm{, }  \\
	f^{4k+5}(z) & \in & [c_{2k+2},c_{2k+3}]\times[c_{2k+2},c_{2k+1}] \textrm{, } \\
	f^{4k+6}(z) & \in & [c_{2k+1}c_{2k+3}+c,c_{2k+3}]\times[c_{2k+2},c_{2k+3}]\subset \\
	            &     & [c_{2k+2},c_{2k+3}]\times[c_{2k+2},c_{2k+3}]\textrm{, }
	\end{array}
	$$
	\noindent i.e. $f^{4(k+1)+2}(z)\in [c_{2(k+1)},c_{2(k+1)+1}]\times[c_{2(k+1)},c_{2(k+1)+1}]$.
	
	Thus, by the claim, since $(c_{2n})_{n\geq 0}$ is a bounded increasing sequence, it follows that $(c_{2n})_{n\geq 0}$ has limit $l\in\mathbb{R}$. Hence,
	
	\begin{center}
		$\displaystyle l=\lim_{n\to\infty}c_{2n+2}=\lim_{n\to\infty}(c_{2n}^2+c)^2+c=(l^2+c)^2+c$.
	\end{center}
	Therefore, $l\in\{a_1,a_2,b_1,b_2\}$, where $b_1=\frac{-1-\sqrt{-3-4c}}{2}$ and $b_2=\frac{-1+\sqrt{-3-4c}}{2}$.
	
	Since $-\frac{3}{4}<c<0$, it follows that $b_1,b_2\in\mathbb{C}\setminus\mathbb{R}$ and $a_2>0$. Therefore $l=a_1$ and we are done.
\end{proof}

Now, we will concentrate to prove the following proposition.

\begin{lemma}
	Let $-1<c\leq -\frac{3}{4}$. If $z\in Y$, then $\displaystyle\lim_{n\to+\infty}f^n(z)=(a_1,a_1)$. \label{lemaaziai}
\end{lemma}

For the proof of Lemma \ref{lemaaziai}, we need the following.

\begin{definition}
	Let $Z_0, Z_1, Z_2, Z_3,Z_4 \subset R$ be the sets defined by (see Figure \ref{rstuy2})
	$$
	\begin{array}{l}
	Z_0=[a_1,a_1+|a_1|^3]\times [a_1-|a_1|^3,a_1], \\
	Z_1=]a_1+|a_1|^2,0]\times[a_1-|a_1|^2,a_1], \\
	Z_2=[a_1+|a_1|^3,a_1+|a_1|^2]\times[a_1-|a_1|^2,a_1-|a_1|^3], \\
	Z_3=[a_1,a_1+|a_1|^3]\times [a_1-|a_1|^2,a_1-|a_1|^3] \textrm{ and} \\
	Z_4=[a_1+|a_1|^3,a_1+|a_1|^2]\times [a_1-|a_1|^3,a_1].
	\end{array}
	$$
\end{definition}

\begin{figure}[h!]
	\centering
	\includegraphics[scale=0.7]{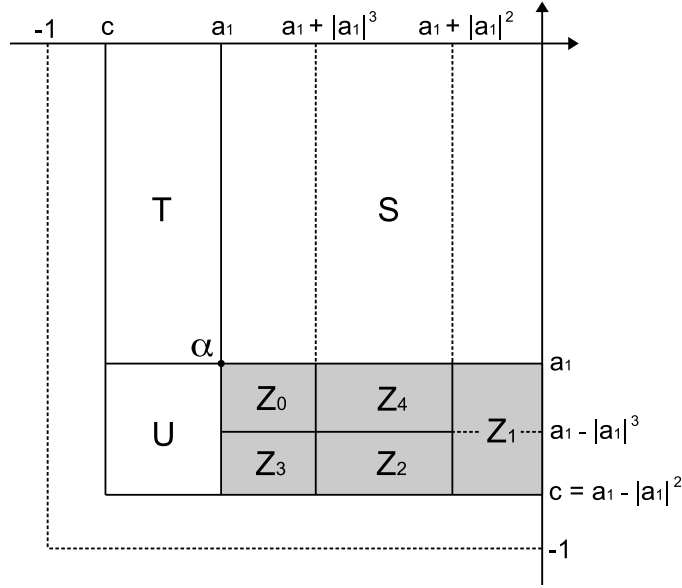}
	\caption{The sets $Z_0$, $Z_1$, $Z_2$, $Z_3$, $Z_4$.} \label{rstuy2}
\end{figure}

The proof of Lemma \ref{lemaaziai} goes as follows. If $z\in Y$, then from Lemma \ref{lema4} it suffices to consider $z\in R$ and it is possible to define the sequence $(n_i)_{i\geq 0}$ where $n_0=0$ and $n_i=\min\{n\in\mathbb{N}: n>{n_{i-1}} \textrm{ and } f^n(z)\in R\}$, for all $i\geq 1$. From this, we will prove that if $z\in R$, then there exists $j\in\mathbb{N}$ such that $f^{n_{j+l}}(z)\in [a_1,a_1+|c|^l|a_1|^3]\times[a_1-|c|^l|a_1|^3,a_1]$, for all $l\geq 0$ and so $\displaystyle\lim_{l\to+\infty}f^{n_{j+l}}(z)=(a_1,a_1)$.

\begin{lemma}
Let $-1<c<0$. If $z\in Z_0$, then $f^{n_i}(z)\in [a_1,a_1+|c|^i|a_1|^3]\times[a_1- |c|^i|a_1|^3,a_1]$, for all nonnegative integer $i\geq 1$.\label{lemaazia0}
\end{lemma}

\begin{proof}
For each set $X\in\{R,S,T,U\}$, let $X(a,b)=\{(x,y)\in X:|x-a_1|\leq a \textrm{ and } |y-a_1|\leq b\}$. Thus, 
\begin{center} 
$R=R(|a_1|,|a_1|^2) \textrm{, } S=S(|a_1|,|a_1|)$, $T=T(|a_1|^2,|a_1|) \textrm{ and } U=U(|a_1|^2,|a_1|^2)$.
\end{center}	
	Let $k\geq 0$ be a nonnegative integer and $z=(x,y)\in R(|c|^k|a_1|^3,|c|^k|a_1|^3)$.

We will prove that $f^{n_1}(z)\in R(|c|^{k+1}|a_1|^3,|c|^{k+1}|a_1|^3)$ and for this, from Lemma \ref{lema4}, we need to consider the following cases:

\noindent $\textbf{(1)}$ $f(z)\in S$,  $f^2(z)\in T$,  $f^3(z)\in U$ and  $f^4(z)\in R$ ($n_1=4$).

\noindent $\textbf{(2)}$ $f(z)\in S$,  $f^2(z)\in T$ and  $f^3(z)\in R$ ($n_1=3$).

\noindent $\textbf{(3)}$ $f(z)\in T$,  $f^2(z)\in U$ and  $f^3(z)\in R$ ($n_1=3$).

\noindent $\textbf{(4)}$ $f(z)\in T$ and $f^2(z)\in R$ ($n_1=2$).

\vspace{1em}

For the proof of Lemma \ref{lemaazia0}, remember that $xy-a_1^2=xy+c-a_1$.

If $(x,y)\in R(|c|^k|a_1|^3,|c|^k|a_1|^3)$, $R=[a_1,0]\times [c,a_1]$, then
	$$\left\{
	\begin{array}{l}
	a_1\leq x\leq a_1+|c|^k|a_1|^3 \\
	a_1-|c|^k|a_1|^3\leq y\leq a_1
	\end{array}
	\right.
	\Longrightarrow -|c|^k|a_1|^4\leq xy-a_1^2\leq |c|^k|a_1|^4,
	$$
\noindent i.e. $f(R(|c|^k|a_1|^3,|c|^k|a_1|^3))\subset S(|c|^k|a_1|^4,|c|^k|a_1|^3)\cup T(|c|^k|a_1|^4,|c|^k|a_1|^3)$.

\vspace{1em} 

If $(x,y)\in S(|c|^k|a_1|^4,|c|^k|a_1|^3)$, $S=[a_1,0]\times [a_1,0]$, then
$$\left\{
\begin{array}{l}
a_1\leq x\leq a_1+|c|^k|a_1|^4 \\
a_1\leq y\leq a_1+|c|^k|a_1|^3
\end{array}
\right.
\Longrightarrow -|c|^k|a_1|^4-|c|^k|a_1|^5+|c|^{2k}|a_1|^7\leq xy-a_1^2\leq 0,
$$

\noindent i.e. $f(S(|c|^k|a_1|^4,|c|^k|a_1|^3))\subset T(|c|^k|a_1|^4+|c|^k|a_1|^5-|c|^{2k}|a_1|^7,|c|^k|a_1|^4)$.

\vspace{1em}

If $(x,y)\in T(|c|^k|a_1|^4+|c|^k|a_1|^5-|c|^{2k}|a_1|^7,|c|^k|a_1|^4)$, $T=[c,a_1]\times [a_1,0]$, then
$$\left\{
\begin{array}{l}
a_1-|c|^k|a_1|^4-|c|^k|a_1|^5+|c|^{2k}|a_1|^7\leq x\leq a_1 \\
a_1\leq y\leq a_1+|c|^k|a_1|^4.
\end{array}
\right. 
$$

\noindent Thus $-|c|^k|a_1|^5\leq xy-a_1^2\leq |c|^k|a_1|^5+|c|^k|a_1|^6-|c|^{2k}|a_1|^8$, i.e. $f( T(|c|^k|a_1|^4+|c|^k|a_1|^5-|c|^{2k}|a_1|^7,|c|^k|a_1|^4)) \subset U(|c|^k|a_1|^5,|c|^k|a_1|^4+|c|^k|a_1|^5-|c|^{2k}|a_1|^7)\cup R(|c|^k|a_1|^5+|c|^k|a_1|^6-|c|^{2k}|a_1|^8,|c|^k|a_1|^4+|c|^k|a_1|^5-|c|^{2k}|a_1|^7)$.

\vspace{1em}

If $(x,y)\in U(|c|^k|a_1|^5,|c|^k|a_1|^4+|c|^k|a_1|^5-|c|^{2k}|a_1|^7)$, $U=[c,a_1]\times [c,a_1]$, then
$$
\begin{array}{rcl}
\left\{
\begin{array}{l}
a_1-|c|^k|a_1|^5\leq x\leq a_1 \\
a_1-|c|^k|a_1|^4-|c|^k|a_1|^5+|c|^{2k}|a_1|^7\leq y\leq a_1 .
\end{array}
\right. \\
\end{array}
$$

\noindent Thus, $0\leq xy-a_1^2\leq |c|^k|a_1|^5+2|c|^k|a_1|^6-|c|^{2k}|a_1|^8+|c|^{2k}|a_1|^9+|c|^{2k}|a_1|^{10}-|c|^{3k}|a_1|^{12}$, i.e. $f(U(|c|^k|a_1|^5,|c|^k|a_1|^4+|c|^k|a_1|^5-|c|^{2k}|a_1|^7))\subset R(|c|^k|a_1|^5+2|c|^k|a_1|^6-|c|^{2k}|a_1|^8+|c|^{2k}|a_1|^9+|c|^{2k}|a_1|^{10}-|c|^{3k}|a_1|^{12},|c|^k|a_1|^5)$.

\vspace{1em}

If $(x,y)\in T(|c|^k|a_1|^4,|c|^k|a_1|^3)$, $T=[c,a_1]\times [a_1,0]$, then
$$\left\{
\begin{array}{l}
a_1-|c|^k|a_1|^4\leq x\leq a_1 \\
a_1\leq y\leq a_1+|c|^k|a_1|^3
\end{array}
\right.
\Longrightarrow -|c|^k|a_1|^4\leq xy-a_1^2\leq |c|^k|a_1|^5,
$$

\noindent i.e. $f(T(|c|^k|a_1|^4,|c|^k|a_1|^3))\subset U(|c|^k|a_1|^4,|c|^k|a_1|^4)\cup R(|c|^k|a_1|^5,|c|^k|a_1|^4)$.

\vspace{1em}

If $(x,y)\in U(|c|^k|a_1|^4,|c|^k|a_1|^4)$, $U=[c,a_1]\times [c,a_1]$, then
$$\left\{
\begin{array}{l}
a_1-|c|^k|a_1|^4\leq x\leq a_1 \\
a_1-|c|^k|a_1|^4\leq y\leq a_1
\end{array}
\right.
\Longrightarrow 0\leq xy-a_1^2\leq 2|c|^k|a_1|^5+|c|^{2k}|a_1|^8,
$$

\noindent i.e. $f( U(|c|^k|a_1|^4,|c|^k|a_1|^4))\subset R(2|c|^k|a_1|^5+|c|^{2k}|a_1|^8,|c|^k|a_1|^4)$.
	
\vspace{1em} 

Hence, we have proved that

\noindent $(a)$ in case $(1)$,  $f^{n_{1}}(z)\in R(|c|^k|a_1|^5+2|c|^k|a_1|^6-|c|^{2k}|a_1|^8+|c|^{2k}|a_1|^9+|c|^{2k}|a_1|^{10}-|c|^{3k}|a_1|^{12}, |c|^k|a_1|^5)=: R(\gamma_1,\lambda_1)$;

\noindent $(b)$ in case $(2)$, $f^{n_{1}}(z)\in  R(|c|^k|a_1|^5+|c|^k|a_1|^6-|c|^{2k}|a_1|^8,|c|^k|a_1|^4+|c|^k|a_1|^5-|c|^{2k}|a_1|^7)=: R(\gamma_2,\lambda_2)$;

\noindent $(c)$ in case $(3)$, $f^{n_{1}}(z)\in R(2|c|^k|a_1|^5+|c|^{2k}|a_1|^8,|c|^k|a_1|^4)=: R(\gamma_3,\lambda_3)$;

\noindent $(d)$ in case $(4)$, $f^{n_{1}}(z)\in R(|c|^k|a_1|^5,|c|^k|a_1|^4)=: R(\gamma_4,\lambda_4)$.
	
If $R'$ is the union of the subsets in $(a)$, $(b)$, $(c)$ and $(d)$, then it is not hard to prove that $R'\subset R(|c|^{k+1}|a_1|^3,|c|^{k+1}|a_1|^3 )$ and we are done. In fact, 
observe that
\begin{center}
$\lambda_1,\lambda_2,\lambda_3,\lambda_4,\gamma_2,\gamma_4\leq |c|^k|a_1|^4+|c|^k|a_1|^5=|c|^k|a_1|^3(a_1^2-a_1)=|c|^{k+1}|a_1|^3$.
\end{center}

Observe also that
\begin{center}
$\gamma_1<|c|^k|a_1|^5+2|c|^k|a_1|^6+|c|^{2k}|a_1|^8(-1-a_1+a_1^2) <|c|^k|a_1|^5+2|c|^k|a_1|^6+|c|^{2k}|a_1|^8(c-a_1+a_1^2)=$  $|c|^k|a_1|^5+2|c|^k|a_1|^6 < 2|c|^k|a_1|^5+|c|^k|a_1|^6 $
\end{center}
and
\begin{center}
	$\gamma_3\leq 2|c|^k|a_1|^5+|c|^k|a_1|^6 $.
\end{center}
On the other hand, we have that
\begin{center}
$2|c|^k|a_1|^5+|c|^k|a_1|^6 \leq |c|^{k+1}|a_1|^3$
\end{center}
if and only if
\begin{center}
$2a_1^2-a_1^3\leq -c$, i.e. $a_1^2+c+a_1^2-a_1^3=a_1+a_1^2-a_1^3\leq 0$,
\end{center}
which is true since $$a_1+a_1^2-a_1^3=-a_1(a_1^2-a_1-1)<-a_1(a_1^2-a_1+c)=0$$.
\end{proof}

\begin{lemma}
	Let $-1<c\leq -\frac{3}{4}$.
	
	\noindent $1.$ If $z\in Z_1$, then there exists $j\geq 1$ such that $f^{n_j}(z)\in Z_0\cup Z_2\cup Z_3\cup Z_4$. \label{lemaazia1}

 	\noindent $2.$ If $z\in Z_2$, then $f^{n_1}(z)\in Z_0\cup Z_3\cup Z_4$.

	\noindent $3.$ If $z\in Z_3$, then $f^{n_1}(z)\in Z_0\cup Z_4$.
	
	\noindent $4.$ If $z\in Z_4$, then there exists $j\geq 1$ such that $f^{n_j}(z)\in Z_0$.
\end{lemma}
\begin{proof} First observe that almost every inclusions in the proof of items $1$, $2$, $3$ and $4$ follows from Lemma \ref{lema4} and from accounts that are easy to be solved  since $-1<c<a_1<0$ and $a_1^2-a_1+c=0$.

\smallskip 

\textbf{1.}	Let $z\in Z_1=]a_1+|a_1|^2,0]\times[a_1-|a_1|^2,a_1]\subset R$. Hence, there exists $|a_1|^2< \varepsilon_0\leq |a_1|$ such that $z\in ]a_1+a_1^2,a_1+\varepsilon_0]\times[a_1-a_1^2,a_1]$. Thus, we have

\smallskip 

\noindent $f(z)\in [a_1+a_1\varepsilon_0,a_1-a_1^4[\times ]a_1+a_1^2,a_1+\varepsilon_0]\subset T$,

\noindent $f^2(z)\in  ](a_1-a_1^4)\varepsilon_0 +a_1-a_1^5 ,(a_1^2+a_1^3)\varepsilon_0 +a_1+a_1^3[\times [a_1+a_1\varepsilon_0,a_1-a_1^4[ \subset U$,

\noindent $f^3(z)\in ](a_1^3+a_1^4-a_1^6-a_1^7)\varepsilon_0+a_1+a_1^4-a_1^6-a_1^7, (a_1^2-a_1^5)\varepsilon_0^2 +$ 

\noindent $(2a_1^2-a_1^5-a_1^6)\varepsilon_0 +a_1-a_1^6[\times ](a_1-a_1^4)\varepsilon_0 +a_1-a_1^5 ,(a_1^2+a_1^3)\varepsilon_0 +a_1+a_1^3[\subset R$.

\smallskip 

Let $h_1:\mathbb{R}\longrightarrow\mathbb{R}$ be the map defined by $$h_1(x)=(a_1^2-a_1^5)x^2 +(2a_1^2-a_1^5-a_1^6)x +a_1-a_1^6.$$
	
	If $f^{n_i}(z)\in Z_1$, for all $i\geq 0$, then $f^{n_i}(z)\in ]a_1+|a_1|^2, a_1+\varepsilon_i[\times]a_1-|a_1|^2,a_1[\subset Z_1$, where $\varepsilon_{i+1}=h_1(\varepsilon_i)$, for all $i\geq 0$.
	
Obviously, $(\varepsilon_i)_{i\geq 0}$ is a decreasing sequence and for each $i$, $\varepsilon_i \in]a_1^2,|a_1|]$.
Hence, there exists $\varepsilon\in [|a_1|^2,|a_1|]$ such that $(\varepsilon_i)_{i\geq 0}$ converges to $\varepsilon$. We thus have $\varepsilon=h_1(\varepsilon)$ which is impossible since $h_1(x)-x<0$ for all $x\in [|a_1|^2,|a_1|]$.
	
	Therefore,  there exists $j\geq 1$ such that $f^{n_j}(z)\in Z_0\cup Z_2\cup Z_3\cup Z_4$.

\smallskip

\textbf{2.} Let $z\in Z_2=[a_1-a_1^3,a_1+a_1^2]\times[a_1-a_1^2,a_1+a_1^3]\subset R$. Thus,

\smallskip

\noindent $f(z)\in ( [a_1+a_1^3+a_1^4+a_1^5, a_1]\times[a_1-a_1^3,a_1+a_1^2]) \cup$

\hfill $( [a_1, a_1-a_1^3-a_1^4+a_1^5]\times [a_1-a_1^3,a_1+a_1^2])\subset T\cup S$.

\smallskip

Hence, if $f(z)\in [a_1+a_1^3+a_1^4+a_1^5, a_1]\times[a_1-a_1^3,a_1+a_1^2] \subset T$, then

\smallskip

\noindent $f^2(z)\in [a_1+a_1^3,a_1+a_1^5-a_1^7-a_1^8]\times[a_1+a_1^3+a_1^4+a_1^5,a_1] \subset U$ and

\noindent $f^3(z)\in [a_1+a_1^6-a_1^8-a_1^9,a_1+2a_1^4+a_1^5+2a_1^6+a_1^7+a_1^8]\times$ 

\hfill $[a_1+a_1^3,a_1+a_1^5-a_1^7-a_1^8]\subset R$

\smallskip

\noindent and if $f(z)\in [a_1, a_1-a_1^3-a_1^4+a_1^5]\times [a_1-a_1^3,a_1+a_1^2] \subset S$, then

\smallskip

\noindent $f^2(z)\in [a_1+a_1^3-a_1^4-2a_1^5+a_1^7,a_1-a_1^4]\times[a_1,a_1-a_1^3-a_1^4+a_1^5] \subset T $,

\smallskip

\noindent $f^3(z)\in [a_1-a_1^4-2a_1^5+a_1^6+a_1^7+a_1^8-a_1^9,a_1+a_1^4-a_1^5-2a_1^6+a_1^8]\times$

\hfill $[a_1+a_1^3-a_1^4-2a_1^5+a_1^7,a_1-a_1^4]\subset$
$ [a_1-a_1^4,a_1+a_1^4-a_1^5-2a_1^6+a_1^8]\times[a_1+a_1^3-a_1^4-2a_1^5+a_1^7,a_1-a_1^4]=$

\smallskip

\noindent $([a_1-a_1^4,a_1]\times[a_1+a_1^3-a_1^4-2a_1^5+a_1^7,a_1-a_1^4])\cup$

\hfill $([a_1,a_1+a_1^4-a_1^5-2a_1^6+a_1^8]\times[a_1+a_1^3-a_1^4-2a_1^5+a_1^7,a_1-a_1^4])\subset U\cup R$

\smallskip

\noindent and if $f^3(z)\in [a_1-a_1^4,a_1]\times[a_1+a_1^3-a_1^4-2a_1^5+a_1^7,a_1-a_1^4]\subset U$, then

\smallskip

\noindent $f^4(z)\in  [a_1-a_1^5,a_1+a_1^4-2a_1^5-2a_1^6-a_1^7+2a_1^8+2a_1^9-a_1^{11}] \times[a_1-a_1^4,a_1]\subset R$.

\smallskip

Hence, if $z\in Z_2$ then $f^{n_1}(z)$ is in one of the following sets:

\noindent $(a)$ \noindent $[a_1+a_1^6-a_1^8-a_1^9,a_1+2a_1^4+a_1^5+2a_1^6+a_1^7+a_1^8]\times$ 

\hfill $[a_1+a_1^3,a_1+a_1^5-a_1^7-a_1^8]=: R_a$,

\noindent $(b)$ $[a_1,a_1+a_1^4-a_1^5-2a_1^6+a_1^8]\times[a_1+a_1^3-a_1^4-2a_1^5+a_1^7,a_1-a_1^4]=:R_b$,

\noindent $(c)$ $[a_1-a_1^5,a_1+a_1^4-2a_1^5-2a_1^6-a_1^7+2a_1^8+2a_1^9-a_1^{11}] \times[a_1-a_1^4,a_1]=: R_c$.

To finish the proof, we need to establish the following claim.

\smallskip 

\textbf{Claim 2:} $R_a\cup R_b\cup R_c\subset Z_0\cup Z_3\cup Z_4$.

\smallskip 
In fact, we have
\begin{center}
$R_a\subset [a_1,a_1+2a_1^4+a_1^5+2a_1^6+a_1^7+a_1^8]\times[a_1+a_1^3,a_1]$
\end{center}
and 
\begin{center}
$a_1+2a_1^4+a_1^5+2a_1^6+a_1^7+a_1^8\leq a_1+2a_1^4+a_1^6\leq a_1+2a_1^4-a_1^5=$ $a_1+a_1^4+ca_1^3\leq a_1+a_1^4-a_1^3=a_1-ca_1^2\leq a_1+a_1^2$,
\end{center}
i.e. $R_a\subset [a_1,a_1+a_1^2]\times[a_1+a_1^3,a_1]\subset Z_0\cup Z_4$.
\smallskip 

We also have that
\begin{center}
$R_b\subset [a_1,a_1+a_1^4-a_1^5-2a_1^6+a_1^8]\times[a_1-a_1^2,a_1]$,
\end{center}
and
\begin{center}
$a_1+a_1^4-a_1^5-2a_1^6+a_1^8\leq a_1+a_1^4-a_1^5= a_1+ca_1^3\leq a_1-a_1^3$,
\end{center}
i.e. $R_b\subset [a_1,a_1-a_1^3]\times[a_1-a_1^2,a_1]\subset Z_0\cup Z_3$.

\smallskip 

We also have that
\begin{center}
	$R_c\subset [a_1,a_1+a_1^4-2a_1^5-2a_1^6-a_1^7+2a_1^8+2a_1^9-a_1^{11}] \times[a_1+a_1^3,a_1]$,
\end{center}
and
\begin{center}
$a_1+a_1^4-2a_1^5-2a_1^6-a_1^7+2a_1^8+2a_1^9-a_1^{11}\leq a_1+a_1^4-2a_1^5+a_1^8= a_1+ca_1^3-a_1^5+a_1^8 \leq a_1-a_1^3-a_1^5+a_1^6 =a_1-a_1^3-ca_1^4\leq a_1-a_1^3+a_1^4=a_1-ca_1^2\leq a_1+a_1^2$,
\end{center}
i.e. $R_c\subset [a_1,a_1+a_1^2]\times[a_1+a_1^3,a_1]\subset Z_0\cup Z_4$, and the proof of Claim 2 is complete.

\smallskip

The proofs of 3 and 4 are respectively similar to the proofs of 2 and 1.
\end{proof}

\begin{proof}[Proof of Lemma  \ref{lemaaziai}] Let $z\in Y$. From Lemma \ref{lema4} it suffices to consider $z\in R$ and it is possible to define the sequence $(n_i)_{i\geq 0}$ where $n_0=0$ and $n_i=\min\{n\in\mathbb{N}: n>{n_{i-1}} \textrm{ and } f^n(z)\in R\}$, for all $i\geq 1$. Since $R=Z_0\cup Z_1\cup Z_2\cup Z_3\cup Z_4$, it follows from Lemma \ref{lemaazia1} that there exists $j\geq 1$ such that $f^{n_j}(z)\in Z_0$. Thus, putting $f^n(z)=(x_n,y_n)$, for all $n\geq 0$, it follows from Lemma \ref{lemaazia0} that $\max\{|x_{n_{j+i}}-a_1|,|y_{n_{j+i}}-a_1|\}\leq |c|^i|a_1|^3$, for all $i \geq 0$, which conclude the proof of Lemma  \ref{lemaaziai}.
\end{proof}

Now, we proceed to the proof of Proposition \ref{propostion2}. Let $z\in\mathcal{R}\setminus \{p,f(p),f^2(p)\}$. From Lemmas \ref{lema3} and \ref{lema33}, there exists $n_1\in \mathbb{N}$ such that $f^{n_1}(z)\in R_2$. Thus, from Lemma \ref{prop66} there exists an integer $n_2\geq n_1$ such that $f^{n_2}(z)\in Y$ and from Lemmas \ref{lemaazia} and \ref{lemaaziai}, we are done. \hfill $\Box$

\begin{remark}
In \cite{bcm} (the case $0< c<\frac{1}{4}$), the points of $W^s(\alpha)$ converges to $\alpha$ in a  monotonous way. Hence, the case $-1<c<0$ is more difficult and delicate since the points of $W^s(\alpha)$ converges to $\alpha$ in a spiral way, which justifies we study the dynamic of $f$ in a more refined filtration around the point $\alpha$, beyond the filtration used in \cite{bcm}.
\end{remark}

\subsection{Proof of Proposition \ref{proposition3}} \label{subsecwstheta} For the proof of Proposition \ref{proposition3}, we need the following (see the relation (\ref{notacao}) for notations).

\begin{lemma}
	The following results are valid:
$$
\begin{array}{lllll}
1. & f(A)\subset B\cup R_0\cup R_1\cup H_1\cup H_2; & \hspace{.3cm} & 6. & f(F)\subset R_1\cup C; \\
2. & f(B)\subset D\cup R_2;                         &               & 7. & f(G)\subset M\cup B; \\
3. & f(C)\subset N\cup E;                           &               & 8. & f(H_1)\subset B\cup R_0\cup H_1; \\
4. & f(D)\subset E\cup F\cup P;                     &               & 9. & f(H_2)\subset B\cup A\cup L. \\
5. & f(E)\subset R_2\cup R_3\cup G;                 &               &    &
\end{array}$$ \label{lema314}	
\end{lemma}

\begin{proof} $1.$ $f(A)\subset [c,+\infty[\times [0,a_2]\subset B\cup R_0\cup R_1\cup H_1\cup H_2$;
	
	\noindent $2.$ $f(B)\subset ]-\infty,c]\times [-1,0]\subset D\cup R_2$;
	
	\noindent $3.$ $f(C)\subset ]-\infty,c]\times ]-\infty,-1]\subset N\cup E$;
	
	\noindent $4.$ $f(D)\subset [c,+\infty[\times ]-\infty,-1]\subset E\cup F\cup P$;
	
	\noindent $5.$ $f(E)\subset [c,+\infty[\times [-1,0]\subset R_2\cup R_3\cup G$;
	
	\noindent $6.$ $f(F)\subset ]-\infty,c]\times [0,1+c]\subset R_1\cup C$;
	
	\noindent $7.$ $f(G)\subset ]-\infty,c]\times [1+c,+\infty[\subset M\cup B$;
	
	\noindent $8.$ $f(H_1)\subset [c,a_2]\times [1+c,a_2]\subset B\cup R_0\cup H_1$;
	
	\noindent $9.$ $f(H_2)\subset [c,+\infty[\times [a_2,+\infty[\subset B\cup A\cup L$.
\end{proof}

\begin{remark}From Lemmas \ref{lema3} and \ref{lema314} and from Proposition \ref{propinf}, we have that if $f(z)\in A$ (respectively $f(z)\in H_2$), then $z\in  H_2$ (respectively $z\in A$), i.e. $f^{-n-1}(A)\subset f^{-n}( H_2)$ and $f^{-n-1}(H_2)\subset f^{-n}(A)$ for all $n\geq 0$. \label{obsnova}
\end{remark}

\begin{lemma}
We have that $\displaystyle \bigcap_{n=0}^{+\infty}f^{-n}(A\cup H_2)\subset W^s(\theta)$. \label{propa2}
\end{lemma}
\begin{proof}
Let $z\in\bigcap_{n=0}^{+\infty}f^{-n}(A\cup H_2)$.	From Remark \ref{obsnova} it suffices to consider $z=(x_0,y_0)\in A$ such that $f^{2n-1}(x_0,y_0)=(x_{2n-1},y_{2n-1})\in H_2$ and $f^{2n}(x_0,y_0)=(x_{2n},y_{2n})\in A$, for all integer $n\geq 1$. Thus,
	\begin{center}
		$0\leq x_0\leq a_2$, $a_2\leq y_0<+\infty$,
	\end{center}
	\begin{equation}
	a_2-c\leq x_0y_0<+\infty, \textrm{  } -c\leq x_0x_1\leq a_2-c, \textrm{  } a_2-c\leq x_2x_1<+\infty. \label{14}
	\end{equation}
	By relation (\ref{14}), we have that $x_0x_1\leq a_2-c\leq x_0y_0$ and $x_0x_1\leq a_2-c\leq x_2x_1$, i.e. $y_2=x_1\leq y_0$ and $x_0\leq x_2$.
	
	Hence, by induction we deduce that $(x_{2n})_n$ is a convergent increasing sequence and $(y_{2n})_n$ is a convergent decreasing sequence. Let $l=\lim x_{2n}=\lim y_{2n+1}$ and $l'=\lim y_{2n}=\lim x_{2n-1}$. Since $x_{2n+2}=x_{2n+1}y_{2n+1}+c$ and $y_{2n+2}=x_{2n+1}=x_{2n}y_{2n}+c$, it follows that $l=l'l+c$ and $l'=ll'+c$, i.e. $l=l'$. Therefore, we have that $l=l^2+c$, i.e. $l\in\{a_1,a_2\}$. Since $l>0$, it follows that $l=a_2$, i.e. $\lim_{n\to+\infty}f^n(x_0,y_0)=(a_2,a_2)=\theta$.
\end{proof}

\begin{lemma}
	If $z\in H_1\setminus \{(a_2,a_2)\}$, then there exists $N\in\mathbb{N}$ such that $f^N(z)\notin H_1$. \label{proph1}
\end{lemma}
\begin{proof} Let $z=(x_0,y_0)\in H_1\setminus \{(a_2,a_2)\}$. We need to consider three cases.
	
	\noindent \textbf{Case 1:} $1+c \leq x_0<a_2$ and  $0 \leq y_0<a_2$.
	
	\noindent Suppose that $f^n(z)\in H_1$ for all $n\in\mathbb{N}$.
	
	Let $z_0=\max\{x_0,y_0\}<a_2$ and define the sequence $z_1=z_0^2+c$ and $z_n=z_{n-1}z_{n-2}+c$, for all integer $n\geq 2$. Thus, we have that
\begin{center}
		$1+c\leq x_0\leq z_0<a_2$, $0\leq y_0\leq z_0<a_2$,
		
		$1+c\leq x_1\leq \underbrace{z_0^2+c}_{=z_1}<a_2^2+c=a_1$ and $0\leq y_1\leq z_0<a_2$.
\end{center}
Hence, by induction we have that
\begin{center}	
		$1+c\leq x_n=x_{n-1}x_{n-2}+c\leq \underbrace{z_{n-1}z_{n-2}+c}_{=z_n}  <a_2^2+c=a_2 $ \newline and $0\leq y_n\leq z_{n-1}<a_2$,
	\end{center}
	
	\noindent for all integer $n\geq 2$. Furthermore, we have that $z_1=z_0^2+c<z_0$ since $a_1<0<1+c<z_0<a_2$. Thus, $z_2=z_1z_0+c<z_0^2+c=z_1$ and by induction on $k\geq 2$, we have that $z_{k+1}=z_kz_{k-1}+c<z_{k-1}z_{k-2}+c=z_k$.

Hence,
\begin{center}
$0<1+c\leq z_n<z_{n-1}<\ldots <z_1<z_0<a_2$,
\end{center}
 for all positive integer $n$, i.e. $(z_n)_n$ is a positive decreasing sequence converging for some $l\in[1+c,a_2[$. Since $z_n=z_{n-1}z_{n-2}+c$, it follows that $l=l^2+c$, i.e. $l\in\{a_1,a_2\}$. On the other hand, we have that $a_1<0<1+c\leq l\leq z_n<a_2$ for all $n\geq 0$, which is a contradiction.
	
	Therefore, there exists $N\in\mathbb{N}$ such that $f^N(z)\notin H_1$.
	
	\noindent \textbf{Case 2:} $x_0=a_2$ and $0\leq y_0<a_2$.
	
	In this case, $x_1=a_2y_0+c<a_2^2+c=a_2$, $y_1=a_2$, $x_2=x_1a_2+c<a_2^2+c=a_2$ and $y_2<a_2$. Hence, if $(x_2,y_2)\in H_1$, then we are done from Case 1 and if $(x_2,y_2)\notin H_1$, then the result is valid for $N=2$.
	
	\noindent \textbf{Case 3:} $1+c\leq x_0<a_2$ and $y_0=a_2$.
	
	In this case $x_1=x_0a_2+c<a_2^2+c=a_2$ and $y_1<a_2$. Hence, if $(x_1,y_1)\in H_1$, then we are done from Case 1 and if $(x_1,y_1)\notin H_1$,  then the result is valid for $N=1$.
\end{proof}

\begin{lemma}
	We have that $\displaystyle W^s(\theta)\subset \bigcap_{n=0}^{+\infty}f^{-n}(A\cup H_2)$. \label{propa3}
\end{lemma}
\begin{proof}
Let $z\in W^s(\theta)$ and suppose that $f^{k_1} (z)\notin A\cup H_2$ for some $k_1\geq 0$. Thus, from Remark \ref{obsnova}, $f^n (z)\notin A\cup H_2$, for all $n\geq k_1$.

Since $z\in W^s(\theta)$,  from Propositions \ref{propinf} and \ref{propostion2} we have that $f^n(z)\notin L\cup \mathcal{R}$, for all $n\geq 0$.

If $z\in H_1$, then  from Lemma \ref{proph1} there exists $k_2>0$ such that $f^{k_2}(z)\notin H_1$.

Hence, from Proposition \ref{propinf} and Lemmas \ref{lema3} and \ref{lema314} there exists $k_3\geq\max\{k_1,k_2\}$ such that $(x_n,y_n)=f^n(z)\notin [0,+\infty[\times[0,+\infty [$, for all $n\geq 0$, i.e. $\min\{|x_n-a_2|,|y_n-a_2|\}>a_2>1$, for all $n\geq k_3$ which contradicts $z\in W^s(\theta)$.

Therefore, $f^n (z)\in A\cup H_2$, for all $n\geq 0$.
\end{proof}

From Lemmas \ref{propa2} and \ref{propa3}, we conclude the proof of Proposition \ref{proposition3}. \hfil $\Box$

\smallskip 

\subsection{Proof of Proposition \ref{proposition4}} \label{subsecwsp}
Denote by $\mathcal{H}=int (\mathcal{R}) \cup W^s(\theta)$ and let $V:=\mathcal{K}^+\setminus\bigcup_{n=0}^{+\infty}f^{-n}(\mathcal{H})$. It suffices to prove that $V= W^s(p)\cup W^s(f(p))\cup W^s(f^2(p))$.

In fact, first observe that $\{p,f(p),f^2(p)\}\subset V$, that is $V\neq\emptyset$.

Since $\mathcal{R}\setminus \{p,f(p),f^2(p)\} \subset W^s(\alpha)$ (by Proposition \ref{propostion2}) and $\alpha \in int(\mathcal{R})$, it follows that $V\subset A\cup B\cup C\cup D\cup E\cup F\cup G\cup H_1\cup H_2$.

For each $X\in\{A,B,C,D,E,F,G,H_1,H_2\}$, denote by $X'=X\cap V$. From Lemma $\ref{lema314}$, we have (see Figure \ref{fig6})
$$
\begin{array}{lll}
1. \textrm{ } f(A')\subset B'\cup H_1'\cup H_2'; & 4. \textrm{ } f(D')\subset E'\cup F'; &  7. \textrm{ } f(G')\subset B';\\
2. \textrm{ } f(B')\subset D'; & 5. \textrm{ } f(E')\subset G'; & 8. \textrm{ } f(H_1')\subset B'\cup H_1';  \\
3. \textrm{ } f(C')\subset E'; & 6. \textrm{ } f(F')\subset C'; & 9. \textrm{ } f(H_2')\subset B'\cup A'.
\end{array}
$$		
\begin{figure}[!h]
		\centering
		\includegraphics[scale=0.5]{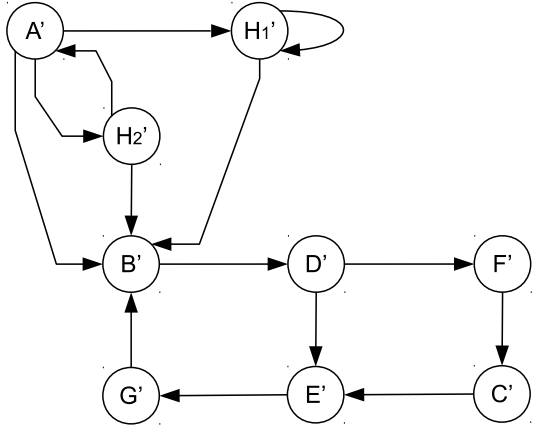}
		\caption{Cycles in the filtration.} \label{fig6}
	\end{figure}
	
	If $z\in A'\cup H_1'\cup H_2'$, then from Lemmas \ref{propa2} and \ref{proph1} there exists $N_1\in\mathbb{N}$ such that $f^{N_1}(z)\in B'$.
	
	If $z\in B'\cup C'\cup D'\cup E'\cup F'\cup G'$, then by Figure \ref{fig6} we can see that there exists $N_2\in\mathbb{N}$ such that $f^{N_2}(z)\in G'$. Thus, if $z\in G'\subset G$, then
\begin{center} 	
	$f(z)\in f(G)\cap B\subset [-1,c]\times[1+c,+\infty[$,
	$f^2(z)\in(]-\infty,c(1+c)+c]\times[-1,c])\cap D\subset (]-\infty.-1]\times[-1,c])$,
	$f^3(z)\in [0,+\infty[\times]-\infty,-1]\cap (E\cup F)\subset F$, $f^4(z)\in C$, $f^5(z)\in E$
\end{center} 
and as before, we have $f^{6n}(z)\in G'$, $f^{6n+1}(z)\in B'$, $f^{6n+2}(z)\in D'$, $f^{6n+3}(z)\in F'$, $f^{6n+4}(z)\in C'$, $f^{6n+5}(z)\in E'$, for all $n\geq 1$.
	
	Therefore, it suffices to study the cycle $G'B'D'F'C'E'G'$ (see Figure \ref{fig6f}). Without loss of generality, consider $z=(x_0,y_0)\in F'$. Thus, for all integer $n\geq 0$, we have $f^{6n}(z)\in F$, $f^{6n+1}(z)\in C$, $f^{6n+2}(z)\in E$, $f^{6n+3}(z)\in G$, $f^{6n+4}(z)\in B$, $f^{6n+5}(z)\in D$. Hence, we have that
	
	\begin{figure}[!h]
		\centering
		\includegraphics[scale=0.5]{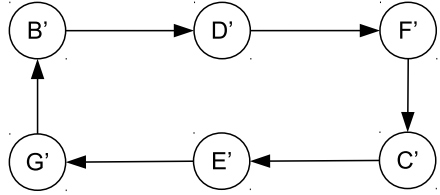}
		\caption{Cycle of the filtration.} \label{fig6f}
	\end{figure}
	\begin{equation}
	0\leq x_0\leq 1+c, \textrm{   } -\infty <y_0\leq -1, \label{17}
	\end{equation}
	\begin{equation}
	-\infty< x_0y_0\leq -1-c, \textrm{   } -1-c \leq x_1x_0\leq -c, \label{18}
	\end{equation}
	\begin{equation}
	-1-c \leq x_1x_0\leq -c, \textrm{   } 1 \leq x_2x_1<+\infty, \label{19}
	\end{equation}
	\begin{equation}
	-1-c \leq x_3x_2\leq -c, \textrm{   } -\infty < x_4x_3\leq-1-c, \label{20}
	\end{equation}
	\begin{equation}
	-c\leq x_5x_4\leq 1, \textrm{   } -\infty <x_6x_5\leq -1-c. \label{21}
	\end{equation}
	
	By relations (\ref{18}) and (\ref{20}), we have that $x_0y_0\leq x_0x_1$ and $x_3x_4\leq x_3x_2$, i.e. $y_0\leq x_1=y_2$ and $x_4\leq x_2\leq 0$.
	
	Since $x_4\leq x_2\leq 0$ and $x_1\leq 0$, it follows that $x_1x_4\geq x_1x_2$. Thus, by relations (\ref{19}) and (\ref{21}), we have that $x_5x_4\leq x_2x_1\leq x_1x_4$ and since $x_4<0$, it follows that $x_1\leq x_5$. Hence, $y_0\leq x_1\leq x_5=y_6$.
	
	By relations (\ref{19}) and (\ref{21}), we have that $x_5x_6\leq x_1x_0$. Since $x_1\leq x_5\leq -1$ and $x_0\geq 0$, it follows that $x_0x_1\leq x_0x_5$. Therefore, $x_5x_6\leq x_1x_0\leq x_0x_5$, i.e. $x_6\geq x_0$.
	
With this, we conclude that $(x_{6n})_n$ and $(y_{6n})_n$ are two bounded increasing sequences.
	
	\noindent \textbf{Claim:} $\displaystyle\lim_{n\to+\infty}f^{6n}(x_0,y_0)=(1+c,-1)=f(p)$.
	
	\noindent In fact, let $\lim_{n\to+\infty} x_{6n}=l$ and $\lim_{n\to+\infty} y_{6n}= l'$. Thus, $f^6(l,l')=(l,l')$. Suppose that $(l,l')\neq (1+c,-1)$ and let $(a,b)=f^6(l,l')$. In the same way that we proved that $x_0\leq x_6$ and $y_0\leq y_6$, we have $l\leq a$ and $l'\leq b$. Since $0\leq l<1+c$ or $-\infty<l'<-1$, we deduce that $l<a$ or $l'<b$, i.e. $f^6(l,l')\neq (l,l')$, which is an absurd. Therefore, $(l,l')= (1+c,-1)=f(p)$ and we obtain the claim.

Hence, since $(1+c,-1)=f(p)$, $(-1,1+c)=f^2(p)$ and $(-1,-1)=p$ is the only $3-$cycle of $f$, from Claim and continuity of $f$ we have
$$
\begin{array}{l}
\displaystyle\lim_{n\to+\infty}f^{6n}(x_0,y_0)=\lim_{n\to+\infty}f^{6n+3}(x_0,y_0)=(1+c,-1)=f(p), \\
\displaystyle\lim_{n\to+\infty}f^{6n+1}(x_0,y_0)=\lim_{n\to+\infty}f^{6n+4}(x_0,y_0)=(-1,1+c)=f^2(p) \textrm{ and} \\
\displaystyle\lim_{n\to+\infty}f^{6n+2}(x_0,y_0)=\lim_{n\to+\infty}f^{6n+5}(x_0,y_0)=(-1,-1)=p,
\end{array}
$$
i.e.

\begin{center}
$\displaystyle\lim_{n\to+\infty}f^{3n}(x_0,y_0)=f(p)$,
$\displaystyle\lim_{n\to+\infty}f^{3n+1}(x_0,y_0)=f^2(p)$ and
$\displaystyle\lim_{n\to+\infty}f^{3n+2}(x_0,y_0)=p$,
\end{center}
which ends the proof of Proposition \ref{proposition4}. \hfill $\Box$

\section{Description of $\mathcal{K}^-$}
\label{secaokmenos}

Now, consider the set $\mathcal{K}^-$ defined by

\smallskip 

\noindent $\mathcal{K}^-:=\left\{(x,y)\in\mathbb{R}^2:f^{-n}(x,y) \textrm{ exists for all } n\in\mathbb{N}\right.$

\hfill $\left. \textrm{ and } (f^{-n}(x,y))_{n\geq 0} \textrm{ is bounded}\right\}$.

\smallskip 

We can see some cases of $\mathcal{K}^-$ in Figure \ref{kmenos}.

\begin{figure}[!h]
	\centering
	\includegraphics[scale=0.29]{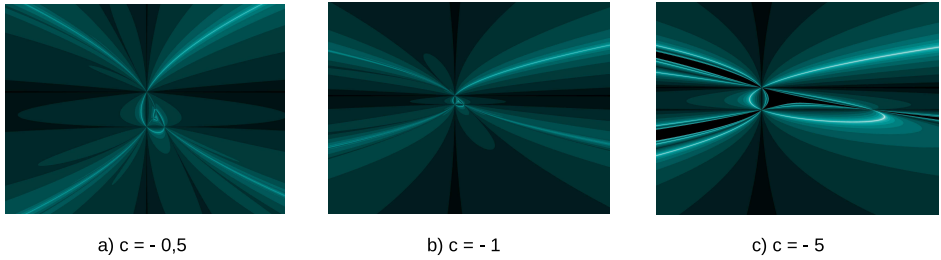}
	\caption{The set $\mathcal{K}^-$ in the cases: $(a)$ $c=-0.5$, $(b)$ $c=-1$, $(c)$ $c=-5$.} \label{kmenos}
\end{figure}

In the same spirit at \cite{bcm} for $0<c<\frac{1}{4}$, we will prove the following Theorem.
	
\begin{theorem} \label{lemaaaziavolta}
	If $-1< c<0$, then $\mathcal{K}^-=W^u(\theta)\cup W^u(p)\cup
	W^u(f(p))\cup W^u(f^2(p))$, where $W^u(x)$ is the unstable manifold of $x$, for all $x\in\{\theta,p,f(p),$ $f^2(p)\}$.
\end{theorem}

For the proof of Theorem \ref{lemaaaziavolta}, we need the following notations (see (\ref{notacao}) for the notations):
\begin{equation} \label{notacaomenos}
\begin{array}{l}
Z'=A\cup B\cup C\cup D\cup E\cup F\cup G\cup H_1\cup H_2, \\
Z=Z'\setminus\{\theta,p,f(p),f^2(p)\}, \\
\mathcal{R}^-=\{z\in \mathcal{R}: f^{-n}(z)\in\mathcal{R}, \textrm{ for all } n\in\mathbb{N}\}.
\end{array}
\end{equation}

We conclude the proof of Theorem \ref{lemaaaziavolta} from the following proposition.

\begin{proposition} \label{propkmenos} 
If $-1<c<0$, then the following assertions holds:

\noindent $1.$ $Z\cap \mathcal{K}^{-}\subset W^u(\theta)$;

\noindent $2.$ $\mathcal{R}^-\subset W^u(p)\cup W^u(f(p))\cup W^u(f^2(p))$;

\noindent $3.$ $L\cap \mathcal{K}^-\subset W^u(\theta)$;

\noindent $4.$ $(M\cup N\cup P)\cap \mathcal{K}^-\subset W^u(\theta)\cup W^u(p)\cup W^u(f(p))\cup W^u(f^2(p))$.
\end{proposition}

The proof of items $1$, $2$, $3$ and $4$ of Proposition \ref{propkmenos} will be displayed in Subsections \ref{umpropkmenos}, \ref{doispropkmenos}, \ref{trespropkmenos} and \ref{quatropropkmenos} respectively.

For the proof of Proposition \ref{propkmenos}, we need the following lemma.

\begin{lemma} If $-1<c<0$, then the following relations are valid:
$$
	\begin{array}{lllll}
	1.	& f^{-1}(A) \subset H_2;                       & \hspace{.2cm}  & 10. & f^{-1}(R_0) \subset  (R_0\cup A\cup H_1);        \\
	2.	& f^{-1}(B) \subset (A\cup G\cup H_1\cup H_2); &           & 11. & f^{-1}(R_1) \subset  (R_0\cup R_3\cup A\cup F);  \\
	3.	& f^{-1}(C) \subset F;                         &                           & 12. & f^{-1}(R_2) \subset  (R_1\cup R_2\cup B\cup E);  \\
	4.	& f^{-1}(D) \subset B;                         &                           & 13. & f^{-1}(R_3) \subset  (R_2\cup E);                \\
	5.	& f^{-1}(E) \subset (C\cup D);                 &                       & 14. & f^{-1}(L) \subset    (L\cup H_2);                \\
	6.	& f^{-1}(F) \subset D;                         &                            & 15. & f^{-1}(M) \subset    (P\cup G);                  \\
	7.	& f^{-1}(G) \subset E;                         &                            & 16. & f^{-1}(N) \subset    (M\cup C);                  \\
	8.	& f^{-1}(H_1) \subset (A\cup H_1);             &                     & 17. & f^{-1}(P) \subset    (N\cup D).                  \\
	9.	& f^{-1}(H_2) \subset A;                       &                           &                   \\
	\end{array}
	$$ \label{prop317}
\end{lemma}
\begin{proof} It follows from item $1$ of Proposition \ref{propinf} and
	Lemmas \ref{lema3} and \ref{lema314}.
\end{proof}

\subsection{Proof of item 1 of Proposition \ref{propkmenos}} \label{umpropkmenos}
From Lemma \ref{prop317}, we have $f^n(Z')\subset f^{n+1}(Z')$ for all $n\geq 0$, i.e. if $f^{-n}(x)\in Z'$, then $f^{-n-i}(x)\in Z'$, for all $i\geq 0$. 
Thus, we can see the dynamic of $f^{-1}$ on $Z'$ in Figure \ref{cicloinversa}.

\begin{figure}[!h]
	\centering
	\includegraphics[scale=0.5]{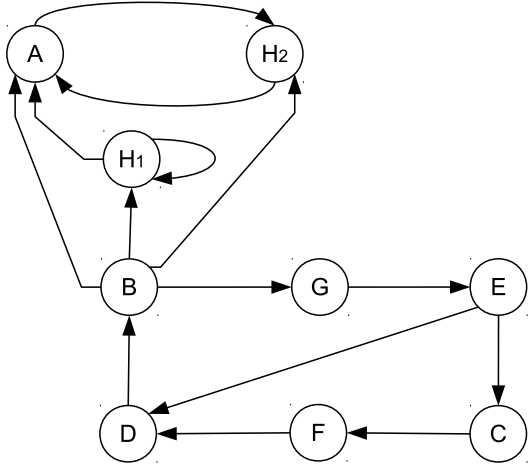}
	\caption{Cycle of filtration $Z'$ by $f^{-1}$.}
	\label{cicloinversa}
\end{figure}

For the proof of item $1$, we will consider three cases:

\noindent \textbf{Case 1:} Let $z\in Z$ such that $f^{-n}(z)\in Z\setminus
(A\cup H_1\cup H_2)$, for all $n\geq 0$. Thus, $z\in
\mathbb{R}^2\setminus \mathcal{K}^-$.

In fact, in this case there exists $i\in\mathbb{N}$, such that $f^{-i}(z)\in
E$.

Without loss of generality, supposes that $z\in E$.

If $f^{-1}(z)\in D$, then $f^{-2}(z)\in B$ and $f^{-3}(z)\in G$. Let
$w=f^{-3}(z)\in G$. Hence, by construction we have $w\in G$,
$f(w)\in B$, $f^2(w)\in D$ and $f^3(w)\in E$, i.e.

\begin{center}
	$w \in [1+c,+\infty[\times[-1,0]=G \textrm{, }  f(w)\in
	(f(G)\cap B)\subset [-1,c]\times[1+c,+\infty[$  and
	$f^2(w)\in (f(f(G)\cap B))\cap D)\subset ]-\infty,-1[\times[-1,c]$,
\end{center}

\noindent implying that $f^3(w)\in ]0,+\infty[\times]-\infty,-1[$, i.e. $f^3(w)\notin E$ which is a contradiction.

Therefore, in this case it suffices to study the cycle $FDBGECF$, i.e. $f^{-6n}(z)\in F$,
$f^{-6n-1}(z)\in D$, $f^{-6n-2}(z)\in B$,  $f^{-6n-3}(z)\in G$,
$f^{-6n-4}(z)\in E$ and $f^{-6n-5}(z)\in C$, for all $n\geq 0$. Hence, from the proof of Claim of Proposition \ref{proposition4}, we have
$f^{-6n}(z)=(x_{-6n},y_{-6n})\in F$ for all $n\geq 0$, where
$(x_{-6n})_{n\geq 0}$ and $(y_{-6n})_{n\geq 0}$ are decreasing sequences.

Let $l=\displaystyle\lim_{n\to +\infty}x_{-6n}$ and
$l'=\displaystyle\lim_{n\to +\infty}y_{-6n}$. Thus,
$f^{-6}(l,l')=(l,l')$, i.e. $(l,l')=(1+c,-1)$, which is an absurd, since $(x_{-6n})_{n\geq 0}$ and $(y_{-6n})_{n\geq 0}$ are
decreasing sequences with $0<x_{-6n}<1+c$ and $y_{-6n}<-1$, for all $n\geq 1$. Therefore,
$$
\begin{array}{lcl}
\displaystyle\lim_{n\to +\infty}f^{-6n}(z)=(0,-\infty), & \hspace{.5cm} & \displaystyle\lim_{n\to +\infty}f^{-6n-3}(z)=(+\infty,0), \\
\displaystyle\lim_{n\to +\infty}f^{-6n-1}(z)=(-\infty,0), & \hspace{.5cm} & \displaystyle\lim_{n\to +\infty}f^{-6n-4}(z)=(0,-\infty), \\
\displaystyle\lim_{n\to +\infty}f^{-6n-2}(z)=(0,+\infty), & \hspace{.5cm} & \displaystyle\lim_{n\to +\infty}f^{-6n-5}(z)=(-\infty,0). \\
\end{array}
$$

\noindent \textbf{Case 2:} Let $z\in H_1$ such that $f^{-n}(z)\in H_1$, for all $n\geq 0$.  Thus $z\in W^u(\theta)$.

In fact, observe that $H_1=([1+c,a_2]\times [0,1+c[)\cup ([1+c,a_2]\times
[1+c,a_2])$ and
\begin{center}
	$(f^{-1}([1+c,a_2]\times [0,1+c[)\cap H_1)\subset (([0,1+c[\times
	[\frac{1}{1+c},+\infty[)\cap H_1)=\emptyset$.
\end{center}
Hence, $\{z\in H_1: f^{-1}(z)\in H_1\}\subset ([1+c,a_2]\times
[1+c,a_2])$.

On the other hand, we have

\noindent $(f^{-1}([1+c,a_2]\times [1+c,a_2])\cap H_1)\subset ([1+c,a_2]\times
	[\frac{1}{a_2},\frac{a^2_2}{1+c}])\cap H_1\subset$

\hfill $[1+c,a_2]\times[\frac{1}{a_2},a_2]$,

\noindent $(f^{-1}([1+c,a_2]\times [\frac{1}{a_2},a_2])\cap H_1)\subset
	([\frac{1}{a_2},a_2]\times [\frac{1}{a_2},a^3_2])\cap H_1\subset
	[\frac{1}{a_2},a_2]\times[\frac{1}{a_2},a_2]$,

\noindent $(f^{-1}([\frac{1}{a_2},a_2]\times[\frac{1}{a_2},a_2])\cap
	H_1)\subset ([\frac{1}{a_2},a_2]\times[\frac{1-ca_2}{a_2^2},a_2^3
	])\cap H_1\subset$

\hfill $([\frac{1}{a_2},a_2]\times[\frac{1-ca_2}{a_2^2},a_2])$.

\noindent $(f^{-1}([\frac{1}{a_2},a_2]\times[\frac{1-ca_2}{a_2^2},a_2])\cap
	H_1)\subset ([\frac{1-ca_2}{a_2^2},a_2]
	\times[\frac{1-ca_2}{a_2^2},\frac{a_2^4}{1-ca_2} ])\cap H_1\subset$

\hfill $([\frac{1-ca_2}{a_2^2},a_2]\times[\frac{1-ca_2}{a_2^2},a_2])$

\noindent and continuing in this way, since $f^{-2n}(z)=(x_{-2n},y_{-2n})$, we have that
\begin{center}
	$\dfrac{1}{a_2^{n}}\displaystyle \left(1-c\sum_{i=1}^{n-1} a_2^i\right)\leq
	x_{-2n},y_{-2n}\leq a_2$, for all $n\geq 2$.
\end{center}

On the other hand,

\begin{center}
	$\dfrac{1}{a_2^{n}}\displaystyle \left(1-c\sum_{i=1}^{n-1}
	a_2^i\right)=\dfrac{1}{a_2^{n}}-\dfrac{c}{a_2^{n}}\cdot
	\dfrac{a_2^{n}-a_2}{a_2-1} \xrightarrow{n\to+\infty}\dfrac{-c}{a_2-1}=a_2$.
\end{center}

Hence,
\begin{center}
	$\displaystyle \lim_{n\to +\infty}f^{-2n}(z)=\lim_{n\to
		+\infty}(x_{-2n},y_{-2n})=(a_2,a_2)=\theta$ and
	$\displaystyle \lim_{n\to +\infty}f^{-2n+1}(z)=\displaystyle \lim_{n\to +\infty}f(x_{-2n},y_{-2n})=(a_2,a_2)=\theta$,
\end{center}
i.e.
\begin{center}
	$\displaystyle \lim_{n\to +\infty}f^{-n}(z)=\lim_{n\to
		+\infty}(x_{-n},y_{-n})=\theta$.
\end{center}
\noindent \textbf{Case 3:} If $z\in A\cup H_2$, then $z\in
\mathbb{R}^2\setminus \mathcal{K}^-$.

Without loss of generality, let $z\in A$. Hence, from Lemma
\ref{prop317}, we have $f^{-2n}(z)=(x_{-2n},y_{-2n})\in A$ and
$f^{-2n-1}(z)=(x_{-2n-1},y_{-2n-1})\in H_2$, for all $n\geq 0$.

From the proof of Lemma \ref{propa2}, it follows that
$(x_{-2n})_{n\geq 0}$ is a strictly decreasing sequence and
$(y_{-2n})_{n\geq 0}$ is a strictly increasing sequence.

Supposes that
\begin{center}
$\displaystyle l=\lim_{n\to +\infty}
x_{-2n}=\lim_{n\to +\infty} y_{-2n+1}$ and $\displaystyle
l'=\lim_{n\to +\infty} y_{-2n}=\lim_{n\to+\infty} x_{-2n-1}$.
\end{center}

Since $(x_{-2n},y_{-2n})=((x_{-2n-2}y_{-2n-2}+c)x_{-2n-2}+c,
x_{-2n-2}y_{-2n-2}+c)$, it follows that $(l,l')=((ll'+c)l+c,ll'+c)$, i.e. $l'=l\in\{a_1,a_2\}$.
Hence, we have that $\displaystyle \lim_{n\to +\infty}
x_{-2n}=\lim_{n\to +\infty}
y_{-2n}\in\{a_1,a_2\}$, which is a contradiction since $(x_{-2n})_{n\geq 0}$ is decreasing, $(y_{-2n})_{n\geq 0}$ is increasing and

\begin{center}
	$a_1<0\leq x_{-2n}<a_2$ e $a_1<a_2< y_{-2n}$, for all $n\geq 0$.
\end{center}

Therefore,

\begin{center}
	$\displaystyle \lim_{n\to+\infty}f^{-2n}(z)=(0,+\infty)$ and
	$\displaystyle \lim_{n\to+\infty}f^{-2n-1}(z)=(+\infty,0)$,
\end{center}

\noindent i.e. $z\in\mathbb{R}^2\setminus \mathcal{K}^-$, which ends the proof of item $1$ of Proposition \ref{propkmenos}. \hfill $\Box$

\subsection{Proof of item 2 of Proposition \ref{propkmenos}} \label{doispropkmenos}
From Lemma \ref{prop317}, we have $f^{n+1}(\mathcal{R}^-)\subset f^{n}(\mathcal{R}^-)$ for all $n\geq 0$, i.e. if $f^{-n-1}(x)\in\mathcal{R}^-$, then $f^{-n}(x)\in\mathcal{R}^-$. Thus, we can see the dynamic of $f^{-1}$ on $\mathcal{R}^-$ in Figure \ref{figrmenos}.
\begin{figure}[!h]
	\centering
	\includegraphics[scale=0.55]{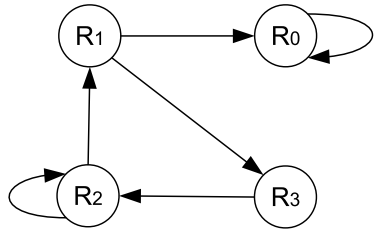}
	\caption{Cycle of filtration $\mathcal{R}^-$ by $f^{-1}$.}
	\label{figrmenos}
\end{figure}

For to prove the item $2$, we need the following lemmas.

\begin{lemma} \label{lemarzeromenos}
	If $-1<c<0$, then there exists $k\in\mathbb{N}$ such that
	$f^{-n}(R_0)\cap R_0=\emptyset$, for all $n\geq k$.
\end{lemma}
\begin{proof} Since $R_0=[0,1+c]\times [0,a_2]$, it follows that

\noindent $(f^{-1}(R_0)\cap R_0)\subset ([0,a_2]\times [\frac{-c}{a_2},+\infty[)\cap R_0\subset [0,1+c]\times[\frac{-c}{a_2},a_2]$,

\noindent $(f^{-1}([0,1+c]\times [\frac{-c}{a_2},a_2])\cap R_0)\subset ([\frac{-c}{a_2},a_2]\times [\frac{-c}{a_2},\frac{a_2}{-c}])\cap R_0\subset$
		
\hfill $[\frac{-c}{a_2},1+c]\times[\frac{-c}{a_2},a_2]$,
		
\noindent $(f^{-1}([\frac{-c}{a_2},1+c]\times[\frac{-c}{a_2},a_2])\cap R_0)\subset ([\frac{-c}{a_2},a_2]\times[\frac{-c-ca_2}{a_2^2},\frac{a_2}{-c} ]\cap R_0)$

\hfill $\subset [\frac{-c}{a_2},1+c]\times[\frac{-c-ca_2}{a_2^2},a_2]$,

\noindent $(f^{-1}([\frac{-c}{a_2},1+c]\times[\frac{-c-ca_2}{a_2^2},a_2])\cap R_0)\subset ([\frac{-c-ca_2}{a_2^2},a_2] \times[\frac{-c-ca_2}{a_2^2},\frac{a_2^2}{-c-ca_2} ]\cap R_0) \subset$

\hfill $[\frac{-c-ca_2}{a_2^2},1+c]\times[\frac{c-ca_2}{a_2^2},a_2]$

\noindent and continuing in this way, we have
	\begin{center}
		$f^{-2n}(R_0)\cap R_0\subset \left[\displaystyle
		\dfrac{-c}{a_2^n}\sum_{i=0}^{n-1} a_2^i,1+c\right]\times
		\left[\dfrac{-c}{a_2^n}\displaystyle \sum_{i=0}^{n-1}
		a_2^i,a_2\right]$, for all $n\geq 1$.
	\end{center}
	
	On the other hand,
	
	\begin{center}
		$\dfrac{-c}{a_2^n}\displaystyle \sum_{i=0}^{n-1}
		a_2^i=\dfrac{-c}{a_2^n}\cdot \dfrac{a_2^n-1}{a_2-1}\xrightarrow{n\to+\infty}\dfrac{-c}{a_2-1}=a_2$.
	\end{center}
	
	Since $1+c<a_2$, it follows that there exists $k\in\mathbb{N}$ such that
	
	\begin{center}
		$\dfrac{-c}{a_2^n}\displaystyle \sum_{i=0}^{n-1} a_2^i>1+c$, for all $n\geq k$,
	\end{center}
	
	\noindent i.e. $f^{-j}(R_0)\cap R_0=\emptyset$, for all $j\geq 2k$.
\end{proof}

\begin{lemma}
Let $-1< c<0$. If $z\in Y\setminus \{\alpha\}$, then there exists $N\in\mathbb{N}$ such that $f^{-N}(z)\not\in Y$. \label{propmula}
\end{lemma}
\begin{proof} Let $z\in Y\setminus\{\alpha\}$ and supposes that $f^{-n}(z)\in
	Y$, for all $n\in\mathbb{N}$.
	
	We need to consider two cases:
	
	\textbf{Case 1:} $-\frac{3}{4}<c<0$.
	
Since $z\notin\alpha$, from the proof of Lemma \ref{lemaazia}, we have that there exists $k\geq 1$ such that $z\in Y\setminus [c_{2k},c_{2k+1}]\times[c_{2k},c_{2k+1}]$, where $c_0=c$ and $c_n=c_{n-1}^2+c$, for all $n\geq 1$. Also, we have
	\begin{center}
		$f^{-4n}(z)\in Y\setminus [c_{2k-2n},c_{2k+1-2n}]\times[c_{2k-2n},c_{2k+1-2n}]$, for all $n\in\{0,\ldots,k\}$,
	\end{center}
	i.e. 
	\begin{center}
		$f^{-4k}(z)\in Y\setminus [c,c^2+c]\times [c,c^2+c]$, $f^{-4k-1}(z)\in Y\setminus [c,c^2+c]\times [c,0]$ and \\
		$f^{-4k-2}(z)\in Y\setminus [c,0]\times [c,0]=\emptyset$,
	\end{center} 
	which is an absurd.
	
	\textbf{Case 2:} $-1<c\leq -\frac{3}{4}$.
	
	Since $R=[a_1,0]\times[c,a_1]$, $S=[a_1,0]\times[a_1,0]$,
	$T=[c,a_1]\times[a_1,0]$ and $U=[c,a_1]\times[c,a_1]$, it follows from Lemma
	\ref{lema4} that
	$$
	\begin{array}{ll}
	f^{-1}(R)\cap Y\subset (T\cup U), & f^{-1}(T)\cap Y\subset (R\cup S),  \\
	f^{-1}(S)\cap Y\subset R,  &	f^{-1}(U)\cap Y\subset T.
	\end{array}
	$$
	
	Thus, the dynamics of $f^{-1}$ at $Y$ satisfies the following diagram:
	\begin{figure}[!h]
	\centering
	\includegraphics[scale=0.48]{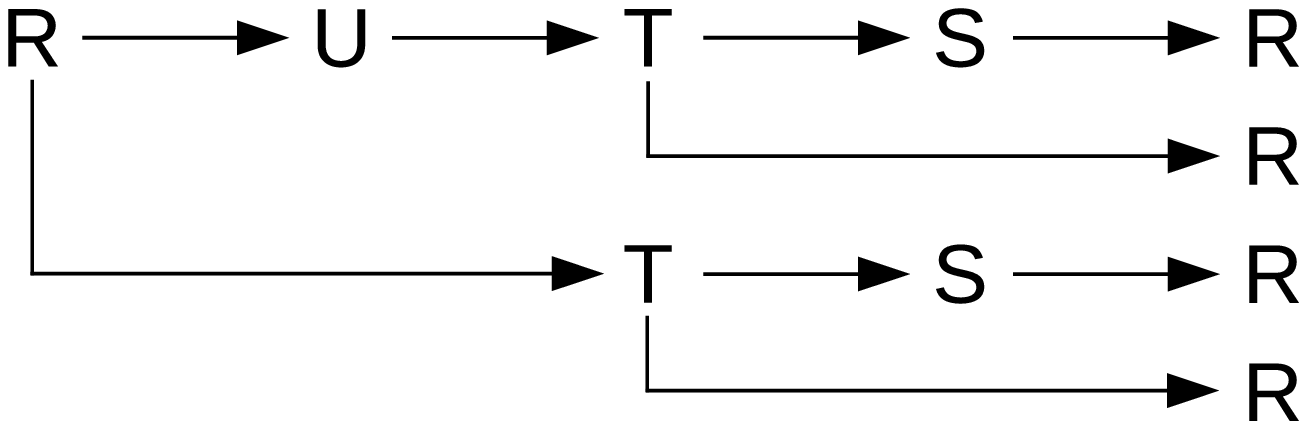}
\end{figure}
	
	Since we are supposing that $f^{-n}(z)\in Y\setminus\{\alpha\}$, for all $n\in\mathbb{N}$, we can supposes without loss of generality that $z\in R\setminus
	\{\alpha\}$. Hence, there exists $k\in\mathbb{N}$, such that $z\in
	R\setminus R(|c|^k|a_1|^3,|c|^k|a_1|^3)$.
	
	Let $n_0=0$ and defines $n_i=\min\{n\in\mathbb{N}:n>n_{i-1} \textrm{ and
	} f^{-n}(z)\in R\}$, for all $i\geq 1$.	Thus, from Lemma \ref{lemaazia0}, we have
	\begin{center}
		$f^{-n_{k}}(z)\in R\setminus R(|a_1|^3,|a_1|^3)\subset Z_1\cup Z_2\cup Z_3\cup Z_4$.
	\end{center}
	
Hence, from Lemmas \ref{lemaazia0} and \ref{lemaazia1}, it follows that there exists $j\geq k$ such that $f^{-n_i}(z)\in Z_1=]a_1+a_1^2,0]\times[c,a_1]$, for all $i\geq j$. Thus, in the same way of the proof of item $1$ of Lemma \ref{lemaazia1}, there exists a bounded increasing sequence $(\varepsilon_i)_{i\geq j}$ in $ ]a_1^2,|a_1|]$ converging to $\varepsilon\in ]a_1^2,|a_1|]$, where $h_1(\varepsilon_i)=\varepsilon_{i-1}$ with $h_1:\mathbb{R}\longrightarrow \mathbb{R}$ defined by $h_1(x)=(a_1^2-a_1^5)x^2+(2a_1^2-a_1^5-a_1^6)x+a_1-a_1^6$. We thus have $\varepsilon=h_1(\varepsilon)$ which is impossible since $h_1(x)-x<0$ for all $x\in ]a_1^2,|a_1|]$.
\end{proof}

\begin{lemma}
	If $-1<c<0$, then the following assertions holds:
	
	\noindent \textbf{1.} If $z\in R_2\setminus Y$, then $f^{-1}(z)\not\in
	R_2$ or $f^{-2}(z)\not\in R_2$.
	
	\noindent \textbf{2.} $f^{-1}(R_3)\cap  R_2\subset [-1,c]\times [-1,c]$.
	
	\noindent \textbf{3.} $f^{-1}([-1,c]\times[-1,c])\subset R_1$. \label{mula}
\end{lemma}
\begin{proof} \textbf{1.} We have that $R_2\setminus Y=([c,0]\times[-1,c[)\cup
	([-1,c[\times[-1,0])$. On the other hand,
	
	\begin{center}
		$f^{-1}([c,0]\times[-1,c[)\subset ([-1,c[\times[-1,0])$ and
		$f^{-1}([-1,c[\times[-1,0])\subset ([-1,0]\times]0,+\infty])$.
	\end{center}
	
	\noindent \textbf{2.} Since $R_3=[0,1+c]\times [-1,0]$, it follows that
	\begin{center}
		$f^{-1}(R_3)\cap R_2\subset([-1,0]\times]-\infty,c])\cap R_2\subset
		[-1,0]\times[-1,c]$.
	\end{center}
However,  $f([c,0]\times[-1,c])\subset [c,0]\times[c,0]$, i.e. $[c,0]\times[-1,c]\not\subset f^{-1}(R_3)$.

Therefore, $f^{-1}(R_3)\cap  R_2\subset [-1,c]\times [-1,c]$.
\begin{center}

\end{center}
	
	\noindent \textbf{3.}
	$f^{-1}([-1,c]\times[-1,c])\subset[-1,c]\times[0,1+c]\subset R_1$
\end{proof}

Now, we proceed to the proof of item $2$ of Proposition \ref{propkmenos}.

Let $z=(x,y)\in \mathcal{R}^-$. From Lemma \ref{lemarzeromenos}, we have $R_0\not\subset \mathcal{R}^-$, i.e. $f^{-n}(z)\in R_1\cup R_2\cup R_3$, for all $n\geq 0$. Hence, from Figure \ref{figrmenos}, Lemma \ref{propmula} and items $1$ and $2$ of Lemma \ref{mula}, we can suppose that $z\in [-1,c]\times[-1,c]$, i.e. we have to study the case
	$f^{-3n}(z)\in [-1,c]\times[-1,c]$, $f^{-3n-1}(z)\in
	[-1,c]\times[0,1+c]$ and $f^{-3n-2}(z)\in R_3$, for all
	$n\in\mathbb{N}$.
	
	Observe that $f([-\sqrt{-c},c]\times[-\sqrt{-c},c])\subset
	[c^2+c,c]\times [-\sqrt{-c},c]\subset R_2$. However, we are considering the points $w\in R_2$ such that $f(w)\in R_3$. Hence, we need to consider $z\in([-1,c]\times[-1,c])\setminus ([-\sqrt{-c},c]\times[-\sqrt{-c},c])$.
	
Let $g:\mathbb{R}\longrightarrow\mathbb{R}$ be the function defined by $g(x)=x(x^2+c)+c$. Since the critical points of $g$ are $-\sqrt{\frac{-c}{3}}$ and $\sqrt{\frac{-c}{3}}$, it is not hard to prove that 
\begin{equation} \label{gcresbije}
g:[-1,-\sqrt{-c}]\longrightarrow [-1,c] \textrm{ is bijective  and increasing.}
\end{equation}

	Let $z_0=-\sqrt{-c}$. Since $z_0\in(-1,c)$, by relation (\ref{gcresbije}) there exists $-1<z_1< z_0$ such that
	$z_0=g(z_1)$. Continuing in this way, we construct a decreasing sequence $(z_n)_n$ such that $-1<z_n=g(z_{n+1})<z_0$, for all $n\geq 1$.

	Let $f^{-i}(x,y)=(x_{-i},y_{-i})$, for all $i\geq 0$. Thus, from the proof of the case 2.1 of
	Lemma \ref{prop66}, we have $(x_{-3n},y_{-3n})\in
	([-1,0]\times[-1,0])\setminus ([z_n,0]\times [z_n,0])$, for all $n\geq 0$.
	
	Since $(z_n)_{n\geq 0}$ is decreasing and $z_0=-\sqrt{-c}<c$, it follows that
	$-z_n+c>0$. Hence, for all $n\geq0$, we have:

\smallskip	

	\noindent $(i)$ If $f^{-3n-3}(z)\in [-1,z_{n+1}]\times[z_n,0]$, then

\smallskip

\noindent $f^{-3n-2}(z)\in [c,-z_n+c]\times[-1,z_{n+1}]\cap R_3= [0,-z_n+c]\times[-1,z_{n+1}]$,

\noindent $f^{-3n-1}(z)\in [z_n,c]\times[0,-z_n+c] \subset R_1$ and

\noindent $f^{-3n}(z)\in [z_n(-z_n+c)+c,c]\times[z_n,c]\subset R_2$.
	
	On the other hand, $-1<z_n<c$ implies $z_n(-z_n+c)+c\geq z_n$, for all $n\geq 0$. Thus, we have $[z_n(-z_n+c)+c,c]\times[z_n,c]\subset[z_n,0]\times[z_n,0]$, which is an absurd, because  $f^{-3n}(z)\in ([-1,0]\times[-1,0])\setminus ([z_n,0]\times [z_n,0])$, for all $n\geq 0$.
	
	\begin{figure}[!h]
		\centering
		\includegraphics[scale=0.6]{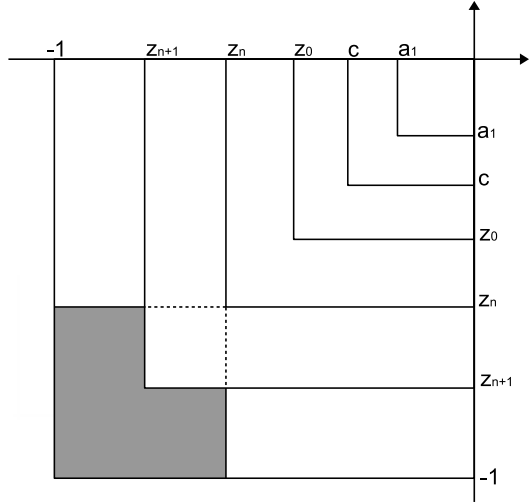}
		\caption{Region that contains $f^{-3n-3}(z)$.}
		\label{figrstuyk}
	\end{figure}
	
	\noindent $(ii)$ If $f^{-3n-3}(z)=(x_{3n+3},y_{3n+3})\in
	[z_n,0]\times[-1,z_{n+1}]$, then
	
\smallskip 
	
\noindent $f^{-3n-2}(z)\in [c,-z_n+c]\times[z_n,0]\cap R_3= [0,-z_n+c]\times[z_n,0]$,

\noindent $f^{-3n-1}(z)\in [z_n(-z_n+c)+c,c]\times[0,-z_n+c] \subset R_1$ and

\noindent $f^{-3n}(z)\in [(-z_n+c)(z_n(-z_n+c)+c)+c,c]\times[z_n(-z_n+c)+c,c]\subset R_2$.

\smallskip

	On the other hand, since $-z_n+c\geq 0$ and $z_n(-z_n+c)+c\geq z_n$, we have

	\begin{center}
		$[(-z_n+c)(z_n(-z_n+c)+c)+c,c]\times[z_n(-z_n+c)+c,c] \subset[z_n,0]\times[z_n,0]$,
	\end{center}
	
	\noindent which is an absurd, since $f^{-3n}(z)\in
	[-1,0]\times[-1,0]\setminus [z_n,0]\times [z_n,0]$, for all $n\geq 0$.
	Therefore, by the items $(i)$ and $(ii)$, we have
	$f^{-3n-3}(z)\in [-1,z_n]\times
	[-1,z_n]\setminus[z_{n+1},z_n]\times[z_{n+1},z_n]$ (the hatching region
	on Figure \ref{figrstuyk}), for all $n\geq 0$. Furthermore, like was done in the proof of case 2.1 of Lemma \ref{prop66}, we have that $(z_n)_n$ converges for $l\in\{-1,a_1,a_2\}$, and since $-1<z_n<c$ for all $n\geq 0$, we have $l=-1$.
	
	Hence, since $f^{-3n-3}(z)\in [-1,z_n]\times [-1,z_n]$ for all $n\geq 0$, it follows that

\noindent $\displaystyle\lim_{n\to+\infty} f^{-3n-3}(z)=p$, $\displaystyle\lim_{n\to+\infty} f^{-3n-2}(z)=f(p)$ and
$\displaystyle\lim_{n\to+\infty} f^{-3n-1}(z)=f^2(p)$, which ends the proof of item $2$ of Proposition \ref{propkmenos}. \hfill $\Box$

\subsection{Proof of item 3 of Proposition \ref{propkmenos}} \label{trespropkmenos}
Let $z=(x,y)\in L\cap \mathcal{K}^-$. From Lemma \ref{prop317} and
by case $3$ of the proof of item $1$ of Proposition \ref{propkmenos}, we have to
suppose that $f^{-n}(z)\in L$, for all $n\geq 0$.

Let $w=\|z\|\geq a_2$ and for each $n\geq 1$, let $b_n:=\dfrac{w}{a_2^n}-\dfrac{c}{a_2^n}\displaystyle\sum_{i=0}^{n-1}a_2^i$, $A_{2n}=[a_2,b_n]\times[a_2,b_{n}]$ and
$A_{2n+1}=[a_2,b_{n}]\times[a_2,b_{n+1}]$.

Observe the sequence of sets $(A_i)_{i\geq 1}$ is well defined, since
\begin{center}
$b_n\geq \dfrac{a_2}{a_2^n}-\dfrac{c}{a_2^n}\displaystyle\sum_{i=0}^{n-1}a_2^i=\dfrac{a_2}{a_2^n}-\dfrac{c}{a_2^n}\cdot \dfrac{a_2^n-1}{a_2-1} =\dfrac{a_2}{a_2^n}+a_2\cdot\dfrac{a_2^n-1}{a_2^n}=a_2$.
\end{center}

\noindent \textbf{Claim:} $f^{-j}(z)\in A_j$, for all $j\geq 2$.

Before to prove the claim, observe that
\begin{center}
$\dfrac{b_n-c}{a_2}=\dfrac{w}{a_2^{n+1}}-\dfrac{c}{a_2^{n+1}}\displaystyle\sum_{i=0}^{n-1}a_2^i-\dfrac{c}{a_2}=\dfrac{w}{a_2^{n+1}}-\dfrac{c}{a_2^{n+1}}\displaystyle\sum_{i=0}^{n}a_2^i=b_{n+1}$,
\end{center}
 for all $n\geq 1$. Thus, since $z\in[a_2,w]\times[a_2,w]$, it follows that

\smallskip

\noindent	$f^{-1}(z)\in\left([a_2,w]\times\left[\frac{a_2-c}{w},\frac{w-c}{a_2}\right]\right)\cap L\subset [a_2,w]\times\left[a_2,b_1\right]$,

\noindent	$f^{-2}(z)\in\left(\left[a_2,b_1\right]\times\left[ \frac{a_2-c}{b_1},\frac{w-c}{a_2}\right]\right)\cap
	L\subset \left[a_2,b_1\right]\times\left[ a_2,b_1\right]=A_2$ and

\noindent	$f^{-3}(z)\in\left(\left[a_2,b_1\right]\times\left[ \frac{a_2-c}{b_1},\frac{b_1-c}{a_2}\right]\right)\cap
	L\subset \left[a_2,b_1\right]\times\left[ a_2,b_2\right]=A_3$.

Let $k\geq 3$ and supposes that $f^{-j}(z)\in A_j$, for all
$j\in\{2,\ldots,k\}$.

\smallskip

$(a)$ If $k=2n$, for some $n\in\mathbb{N}$, then

\smallskip

\noindent $f^{-2n}(z)\in A_{2n}=[a_2,b_n]\times[ a_2,b_{n}]$ and

\noindent $f^{-2n-1}(z)\in\left(\left[a_2,b_{n}\right]\times\left[ \frac{a_2-c}{b_{n}},\frac{b_n-c}{a_2}\right]\right)\cap L
 \subset [a_2,b_{n}]\times [ a_2,b_{n+1}]=A_{2n+1}$.

\smallskip

$(b)$ If $k=2n+1$, for some $n\in\mathbb{N}$, then

\smallskip

\noindent $f^{-2n-1}(z)\in A_{2n+1}=[a_2,b_{n}]\times[ a_2,b_{n+1}]$ and 

\noindent $ f^{-2n-2}(z)\in\left(\left[a_2,b_{n+1}\right]\times\left[ \frac{a_2-c}{b_{n+1}},\frac{b_{n}-c}{a_2}\right]\right)\cap L
 \subset [a_2,b_{n+1}]\times [ a_2,b_{n+1}]=A_{2n+2}$.

Therefore, by induction we proved that $f^{-j}(z)\in A_j$, for all $j\geq 2$.

On the other hand, observe that
\begin{center}
	$b_n= \dfrac{w}{a_2^n}-\dfrac{c}{a_2^n}\displaystyle\sum_{i=0}^{n-1}a_2^i=\dfrac{w}{a_2^n}-\dfrac{c}{a_2^n}\cdot \dfrac{a_2^n-1}{a_2-1} =\dfrac{w}{a_2^n}+a_2\cdot\dfrac{a_2^n-1}{a_2^n}\xrightarrow{n\to+\infty} a_2$.
\end{center}
Hence, $\displaystyle\lim_{n\to+\infty}f^{-n}(z)=(a_2,a_2)=\theta$ and the proof of item $3$ of Proposition \ref{propkmenos} is done. \hfill $\Box$

\subsection{Proof of item 4 of Proposition \ref{propkmenos}} \label{quatropropkmenos}
 Let $z=(x,y)\in (M\cup N\cup P)\cap \mathcal{K}^-$. Thus, from Lemma \ref{prop317} and 
item $1$ of Proposition \ref{propkmenos}, it suffices to consider the case where $f^{-n}(z)=(x_{-n},y_{-n})\in(M\cup N\cup P)$, for all $n\geq 0$.
	
	Without loss of generality, supposes that $z\in N$. Hence,
	\begin{center}
$(x_{-3n},y_{-3n})\in N$, $(x_{-3n-1},y_{-3n-1})\in M$ and $(x_{-3n-2},y_{-3n-2})\in P$,
	\end{center}
	\noindent for all $n\geq 0$.
	
	\noindent \textbf{Claim:} $x_{-3n}\leq x_{-3n-3}$ and $y_{-3n}\leq y_{-3n-3}$, for all $n\geq 0$.
	
	In fact, let $n\geq 0$. Thus,
\begin{center}
$max\{x_{-3n-3},y_{-3n-3}\}\leq-1$ and $x_{-3n-3}y_{-3n-3}+c=x_{-3n-2}\geq 1+c>0$, 
\end{center}	
which implies
\begin{center}
$y_{-3n}=x_{-3n-3}(x_{-3n-3}y_{-3n-3}+c)+c\leq -(x_{-3n-3}y_{-3n-3}+c)+c\leq y_{-3n-3}$.
\end{center}

On the other hand,
\begin{center}
$x_{-3n}=y_{-3n}(x_{-3n-3}y_{-3n-3}+c)+c \leq y_{-3n-3}(x_{-3n-3}y_{-3n-3}+c)+c\leq -(x_{-3n-3}y_{-3n-3}+c)+c\leq x_{-3n-3}$,
\end{center}
 which concludes the proof of claim.
	
	Hence, since the sequences $(x_{-3n})_{n\geq 0}$ and $(y_{-3n})_{n\geq 0}$ are not decreasing and $max\{x_{-3n},y_{-3n}\}\leq-1$, for all $n\geq 0$, it follows that  $(x_{-3n})_{n\geq 0}$ and $(y_{-3n})_{n\geq 0}$ are convergent. Let
	
	\begin{center}
		$l:=\displaystyle\lim_{n\to+\infty}x_{-3n}\leq-1$ \hspace{.5cm} and \hspace{.5cm} $l':=\displaystyle\lim_{n\to+\infty}y_{-3n}\leq -1$.
	\end{center}
Since
\begin{center}
$y_{-3n+3}=(x_{-3n}y_{-3n}+c)x_{-3n}+c$ and $x_{-3n+3}=y_{-3n+3}(x_{-3n}y_{-3n}+c)+c$,
\end{center}	
for all $n\geq 0$, we have
	
	\begin{center}
		$l'=(l'l+c)l+c$ and $l=l'(ll'+c)+c$, which implies $(l-l')(l'l+c+1)=0$.
	\end{center}
	
	Since $l'l+c+1>0$, it follows that $l=l'$. Thus, $l^3+(c-1)l+c=0$, i.e. $l\in\{-1,a_1,a_2\}$ and so $l=l'=-1$.
	
	Therefore, we have that
	\begin{center}
		$\displaystyle\lim_{n\to+\infty} f^{-3n}(z)=(-1,-1)=p$,
		$\displaystyle\lim_{n\to+\infty} f^{-3n-1}(z)=(-1,1+c)=f^2(p)$ \newline and
		$\displaystyle\lim_{n\to+\infty} f^{-3n-2}(z)=(1+c,-1)=f(p)$,
	\end{center}
which ends the proof of item $4$ of Proposition \ref{propkmenos}. \hfill $\Box$

\section*{Acknowledgments} I heartily thank \emph{Ali Messaoudi} for many helpful comments, discussions and suggestions. I also want to thank the \emph{Institut de Math\'ematiques de Luminy}, where I have done part of this work. I would also like to thank the anonymous referees for an especially careful reading of this paper and for the useful comments and remarks.


\begin{thebibliography}{99}

\bibitem{bs}
\newblock E. Bedford and J. Smillie,
\newblock Polynomial diffeomorphisms of $\mathbb{C}^2$,
\newblock \emph{VI. Connectivity of J. Ann. of Math.}, \textbf{148(2)} (1998), 695--735.

\bibitem{bcm}
\newblock S. Bonnot, A. De Carvalho and A. Messaoudi,
\newblock Julia sets for Fibonacci endomorphisms of $\mathbb{R}^2$,
\newblock \emph{Dynamical Systems}, \textbf{33(4)} (2018), 622--645. 

\bibitem{cmv}
\newblock D. A. Caprio, A. Messaoudi and G. Valle,
\newblock Stochastic adding machines based on Bratteli diagrams,
\newblock \emph{arXiv:1712.09476 [math.DS]}.

\bibitem{dv}
\newblock R. Devaney,
\newblock An Introduction to Chaotic Dynamical Systems,
\newblock 2nd ed. (1989), Addison-Wesley.

\bibitem{abms}
\newblock H. El Abdalaoui, S. Bonnot, A. Messaoudi, and O. Sester,
\newblock On the Fibonacci complex dynamical systems,
\newblock \emph{Discrete and Continuous Dynamical Systems}, \textbf{36(5)} (2016), 2449–2471.

\bibitem{am}
\newblock H. El Abdalaoui and A. Messaoudi,
\newblock On the spectrum of stochastic perturbations of the shift and Julia sets,
\newblock \emph{Fundamenta Mathematicae}, \textbf{218(2)} (2012), 47--68.

\bibitem{fatou19}
\newblock P. Fatou,
\newblock Sur les \'equations fonctionnelles, 
\newblock \emph{Bull. Soc. Mat. France}, \textbf{47} (1919) 161--271; \textbf{48} (1920), 33--94, 208--314.

\bibitem{fatou21}
\newblock  P. Fatou,
\newblock  Sur les fonctions qui admettent plusieurs th\'eor\`emes de multiplication,
\newblock  \emph{C.R.A.S.}, \textbf{173} (1921), 571--573.

\bibitem{fs}
\newblock J. E. Fornaess and N. Sibony,
\newblock Fatou and Julia sets for entire mappings in $\mathbb{C}^k$,
\newblock \emph{Math. Ann.}, \textbf{311} (1998), 27-40.

\bibitem{g}
\newblock V. Guedj,
\newblock Dynamics of quadratic polynomial mappings of $\mathbb{C}^2$,
\newblock \emph{Michigan Math. J.}, \textbf{52(3)} (2004), 627--648.

\bibitem{hw}
\newblock S. Hayes and C. Wolf,
\newblock Dynamics of a one-parameter family of H\'enon maps,
\newblock \emph{Dynamical Systems} \textbf{21(4)} (2006), 399--407.

\bibitem{julia18}
\newblock G. Julia,
\newblock  M\'emoire sur l'it\'eration des fonctions rationnelles,
\newblock  \emph{J. Math. Pure Appl.}, \textbf{8} (1918), 47--245.

\bibitem{julia22}
\newblock G. Julia,
\newblock  M\'emoire sur la permutabilit\'e des fractions rationnelles,
\newblock  \emph{Annales scientifiques de l'\'Ecoles Normale Sup\'erieure}, \textbf{39(11)} (1922), 131--215.

\bibitem{killentaylor}
\newblock P. R. Killeen and T. J. Taylor,
\newblock  A stochastic adding machine and complex dynamics,
\newblock  \emph{Nonlinearity}, \textbf{13(6)} (2000), 1889–1903.

\bibitem{ms}
\newblock A. Messaoudi and D. Smania,
\newblock Eigenvalues of stochastic adding machine,
\newblock \emph{Stochastics and Dynamics}, \textbf{10(2)} (2010), 291--313.

\end{thebibliography}
\end{document}